\documentclass{amsip-p-l}
\usepackage{amsmath,amssymb,amsthm,amsfonts,amscd}
\usepackage[numeric]{amsrefs}
\input colordvi

\newcommand{\nc}{\newcommand}
\newcommand{\rc}{\renewcommand}
\nc{\cmt}[1]{\noindent\textRed{{\fbox{W}} #1}\textBlack}
\nc{\cmtk}[1]{\noindent\textBlue{{\fbox{K}} #1}\textBlack}
\nc{\cmtm}[1]{\noindent\textGreen{{\fbox{M}} #1}\textBlack}

\newtheorem{thm}[equation]{Theorem}
\newtheorem{prop}[equation]{Proposition}

\newtheorem{cor}[equation]{Corollary}

\newtheorem{conj}[equation]{Conjecture}

\numberwithin{equation}{section}

\DeclareMathOperator{\Ad}{Ad}

\nc{\on}{\operatorname}

\nc{\al}{{\alpha }}
\nc{\be}{{\beta }}
\nc{\ga}{{\gamma }}
\nc{\de}{{\delta }}
\nc{\del}{{\partial }}
\nc{\ep}{{\varepsilon }}
\nc{\vap}{{\epsilon }}

\nc{\ze}{{\zeta }}
\nc{\et}{{\eta }}
\rc{\th}{{\theta }}

\nc{\vth}{{\vartheta }}

\nc{\io}{{\iota }}
\nc{\ka}{{\kappa }}
\nc{\la}{{\lambda }}
\nc{\vrho}{{\varrho}}
\nc{\si}{{\sigma }}
\nc{\ups}{{\upsilon }}
\nc{\vphi}{{\varphi }}
\nc{\om}{{\omega }}

\nc{\Ga}{{\Gamma }}
\nc{\De}{{\Delta }}
\nc{\nab}{{\nabla}}
\nc{\Th}{{\Theta }}
\nc{\La}{{\Lambda }}
\nc{\Si}{{\Sigma }}
\nc{\Ups}{{\Upsilon }}
\nc{\Om}{{\Omega }}

\renewcommand\a{\alpha}

\renewcommand\l{\lambda}

\newcommand\G{\Gamma}
\renewcommand\L{\Lambda}

\nc{\fA}{{\mathfrak A}}
\nc{\fB}{{\mathfrak B}}
\nc{\fC}{{\mathfrak C}}
\nc{\fD}{{\mathfrak D}}
\nc{\fE}{{\mathfrak E}}
\nc{\fF}{{\mathfrak F}}
\nc{\fG}{{\mathfrak G}}
\nc{\fH}{{\mathfrak H}}
\nc{\fI}{{\mathfrak I}}
\nc{\fJ}{{\mathfrak J}}
\nc{\fK}{{\mathfrak K}}
\nc{\fL}{{\mathfrak L}}
\nc{\fM}{{\mathfrak M}}
\nc{\fN}{{\mathfrak N}}
\nc{\fO}{{\mathfrak O}}
\nc{\fP}{{\mathfrak P}}
\nc{\fQ}{{\mathfrak Q}}
\nc{\fR}{{\mathfrak R}}
\nc{\fS}{{\mathfrak S}}
\nc{\fT}{{\mathfrak T}}
\nc{\fU}{{\mathfrak U}}
\nc{\fV}{{\mathfrak V}}
\nc{\fW}{{\mathfrak W}}
\nc{\fZ}{{\mathfrak Z}}
\nc{\fX}{{\mathfrak X}}
\nc{\fY}{{\mathfrak Y}}
\nc{\fa}{{\mathfrak a}}
\nc{\fb}{{\mathfrak b}}
\nc{\fc}{{\mathfrak c}}
\nc{\fd}{{\mathfrak d}}
\nc{\fe}{{\mathfrak e}}
\nc{\ff}{{\mathfrak f}}
\nc{\fg}{{\mathfrak g}}
\nc{\fh}{{\mathfrak h}}
\nc{\fiI}{{\mathfrak i}}  
\nc{\ffi}{{\mathfrak i}}  
\nc{\fj}{{\mathfrak j}}
\nc{\fk}{{\mathfrak k}}
\nc{\fl}{{\mathfrak{l}}}
\nc{\fm}{{\mathfrak m}}
\nc{\fn}{{\mathfrak n}}
\nc{\fo}{{\mathfrak o}}
\nc{\fp}{{\mathfrak p}}
\nc{\fq}{{\mathfrak q}}
\nc{\fr}{{\mathfrak r}}
\nc{\fs}{{\mathfrak s}}
\nc{\ft}{{\mathfrak t}}
\nc{\fu}{{\mathfrak u}}
\nc{\fv}{{\mathfrak v}}
\nc{\fw}{{\mathfrak w}}
\nc{\fz}{{\mathfrak z}}
\nc{\fx}{{\mathfrak x}}
\nc{\fy}{{\mathfrak y}}

\nc{\bA}{{\mathbb A}}
\nc{\bB}{{\mathbb B}}
\nc{\bC}{{\mathbb C}}
\nc{\bD}{{\mathbb D}}
\nc{\bE}{{\mathbb E}}
\nc{\bF}{{\mathbb F}}
\nc{\bG}{{\mathbb G}}
\nc{\bH}{{\mathbb H}}
\nc{\bI}{{\mathbb I}}
\nc{\bJ}{{\mathbb J}}
\nc{\bK}{{\mathbb K}}
\nc{\bL}{{\mathbb L}}
\nc{\bM}{{\mathbb M}}
\nc{\bN}{{\mathbb N}}
\nc{\bO}{{\mathbb O}}
\nc{\bP}{{\mathbb P}}
\nc{\bQ}{{\mathbb Q}}
\nc{\bR}{{\mathbb R}}
\nc{\bS}{{\mathbb S}}
\nc{\bT}{{\mathbb T}}
\nc{\bU}{{\mathbb U}}
\nc{\bV}{{\mathbb V}}
\nc{\bW}{{\mathbb W}}
\nc{\bZ}{{\mathbb Z}}
\nc{\bX}{{\mathbb X}}
\nc{\bY}{{\mathbb Y}}

\newcommand{\Z}{{\mathbb{Z}}}
\newcommand{\R}{{\mathbb{R}}}

\newcommand{\C}{{\mathbb{C}}}
\newcommand{\D}{{\mathbb{D}}}

\newcommand{\Q}{{\mathbb{Q}}}
\renewcommand{\O}{{\mathbb{O}}}

\nc{\cA}{{\mathcal A}}
\nc{\cB}{{\mathcal B}}
\nc{\cC}{{\mathcal C}}
\nc{\cD}{{\mathcal D}}
\nc{\cE}{{\mathcal E}}
\nc{\cF}{{\mathcal F}}
\nc{\cH}{{\mathcal H}}
\nc{\cI}{{\mathcal I}}
\nc{\cJ}{{\mathcal J}}
\nc{\cK}{{\mathcal K}}
\nc{\cL}{{\mathcal L}}
\nc{\cM}{{\mathcal M}}
\nc{\cN}{{\mathcal N}}
\nc{\cO}{{\mathcal O}}
\nc{\cP}{{\mathcal P}}
\nc{\cQ}{{\mathcal Q}}
\nc{\cR}{{\mathcal R}}
\nc{\cS}{{\mathcal S}}
\nc{\cT}{{\mathcal T}}
\nc{\cU}{{\mathcal U}}
\nc{\cV}{{\mathcal V}}
\nc{\cW}{{\mathcal W}}
\nc{\cZ}{{\mathcal Z}}
\nc{\cX}{{\mathcal X}}
\nc{\cY}{{\mathcal Y}}

\newcommand\cExt{\mathcal E\! \it{xt}}
\renewcommand\O{{\mathcal O}}

\newcommand\codim{\operatorname{codim}}

\newcommand{\rk}{\operatorname{rk}}

\newcommand{\Hom}{\operatorname{Hom}}
\newcommand{\Ker}{\operatorname{Ker}}

\newcommand{\gr}{\operatorname{gr}}
\newcommand{\Mod}{\operatorname{Mod}}
\newcommand{\HC}{\operatorname{HC}}
\newcommand{\HM}{\operatorname{HM}}
\newcommand{\CHM}{\operatorname{CHM}}
\newcommand{\MHM}{\operatorname{MHM}}
\newcommand{\CMHM}{\operatorname{CMHM}}
\renewcommand{\Re}{\operatorname{Re}}
\renewcommand{\Im}{\operatorname{Im}}
\renewcommand\({\left(}
\renewcommand\){\right)}

\newcommand{\gobble}[1]{}
  \newcommand{\rangeref}[2]{%
    \ref{#1}--\afterassignment\gobble\fam 0\ref{#2}%
  }

\begin{document}

\title{Hodge theory and unitary representations of reductive Lie groups}

\author{Wilfried Schmid}
\address{Department of Mathematics, Harvard University, Cambridge, MA 02138}
\email{schmid@math.harvard.edu}
\thanks{The first author was supported in part by DARPA grant HR0011-04-1-0031 and NSF grants DMS-0500922, DMS-1001405.}

\author{Kari Vilonen}\address{Department of Mathematics, Northwestern University, Evanston, IL 60208, also Department of Mathematics, Helsinki University, Helsinki, Finland}
\email{vilonen@northwestern.edu, vilonen@math.helsinki.fi}
\thanks{The second author was supported in part by DARPA grant HR0011-04-1-003, by NSF, and by AFOSR grant FA9550-08-1-0315.}

\date{March 9, 2011.}


\begin{abstract}
We present an application of Hodge theory towards the study of irreducible unitary representations of reductive Lie
groups. We describe a conjecture about such representations and discuss some progress towards its proof.
\end{abstract}

\maketitle

\section{Introduction}\label{sec:intro}

In this paper we state a conjecture on the unitary dual of reductive Lie groups and formulate certain foundational
results towards it. We are currently working towards a proof. The conjecture does not give an explicit description
of the unitary dual. Rather, it provides a strong functorial framework for studying unitary representations. Our
main tool is M. Saito's theory of mixed Hodge modules \cites{saito:rims1990}, applied to the Beilinson-Bernstein
realization of Harish Chandra modules \cite{bb:1993}

Our paper represents an elaboration of the first named author's lecture at the Sanya Inauguration conference. Like
that lecture, it aims to provide an expository account of the present state of our work.

We consider a linear reductive Lie group $G_\R$ with maximal compact subgroup $K_\R$. The complexification $G$ of
$G_\R$ contains a unique compact real form $U_\R$ such that $K_\R = G_\R \cap U_\R$. We denote the Lie algebras of
$G_\R$, $U_\R$, $K_\R$ by the subscripted Gothic letters $\fg_\R$, $\fu_\R$, and $\fk_\R$\,;\,\ the first two of
these lie as real forms in $\fg$, the Lie algebra of the complex group $G$, and similarly $\fk_\R$ is a real form
of $\fk$, the Lie algebra of the complexification $K$ of the group $K_\R$. To each irreducible unitary
representation $\pi$ of $G_\R$, one associates its Harish Chandra module: a finitely generated module $V$ over the
universal enveloping algebra $U(\fg)$, equipped with an algebraic action of $K$ which is both compatible with the
$U(\fg)$ module structure and admissible, in the sense that $\dim\Hom_K(W,V)<\infty$ for each irreducible
$K$-module $W$. The irreducibility of $\pi$ implies the irreducibility of its Harish Chandra module $V$, and the
unitary nature of $\pi$ is reflected by a positive definite hermitian form $(\ ,\ )$ on $V$ that is
$\fg_\R$-invariant, i.e.,
\begin{equation}
\label{infinvariance}
(\zeta u,v)\, + \, (u,\zeta v)\ = \ 0\ \ \ \text{for all}\ \ u, v\in V,\,\ \zeta \in \fg_\R\,.
\end{equation}
Conversely, the completion of any irreducible Harish Chandra module with a positive definite $\fg_\R$-invariant
hermitian form is the representation space of an irreducible unitary representation $\pi$ of $G_\R$ \cite{HC1}. It
should be noted that the $\fg_\R$-invariant hermitian form is unique up to scaling, if it exists at all.

The correspondence between irreducible unitary representations and irreducible Harish Chandra modules with
$\fg_\R$-invariant hermitian form makes it possible to break up the problem of describing the unitary dual
$\widehat G_\R$ into three sub-problems:
\begin{equation}
\begin{aligned}
\label{subproblems}
{\rm a)}\, \ &\text{classify the irreducible Harish Chandra modules;}
\\
{\rm b)}\, \ &\text{single out those carrying a non-zero $\fg_\R$-invariant hermitian form;}
\\
{\rm c)}\, \ &\text{determine if the $\fg_\R$-invariant hermitian form has a definite sign.}
\end{aligned}
\end{equation}
Problem a) has been solved in several different, yet fundamentally equivalent ways \cites{bb:1981, HMSWI, HMSWII,
KZ1982, langlands1989, vogan1981}, and b) can be answered by elementary considerations. That leaves c) as the
remaining -- and difficult -- open problem. This problem has been solved for certain groups, e.g.
\cites{vogan1986,tadic2009}, and certain classes of representations, e.g. \cite{barbasch2009}, but these solutions
do not at all suggest a general answer.

In any irreducible Harish Chandra module $V$, the center $Z(\fg)$ of $U(\fg)$ acts by a character, the so-called
infinitesimal character of $V$. When addressing the problem (\ref{subproblems}c), we may and shall assume that
\begin{equation}
\label{realinfcharacter1}
\text{$V$ has a real infinitesimal character;}
\end{equation}
for the definition of this notion see (\ref{realinfchar1}) below. Indeed, any irreducible unitary Harish Chandra
module is irreducibly and unitarily induced from one that satisfies this hypothesis \cite{knapp1986}*{theorem
16.10}. Vogan and his coworkers have pointed out that the hypothesis (\ref{realinfcharacter1}) on an irreducible
Harish Chandra module $V$ implies the existence of a non-zero $\fu_\R$-invariant hermitian form $(\ ,\ )_{\fu_R}$;
as in the $\fg_\R$-invariant case, the hermitian form is then unique up to scaling. When also a non-zero
$\fg_\R$-invariant hermitian form $(\ , \ )$ exists, then the two are very directly related \cite{vogan2009}. If,
for example, $G_\R$ and $K_\R$ have the same rank, there exist a root of unity $\,c\,$ and a character $\,\psi\,$
from the semigroup\begin{footnote}{with respect to the semigroup structure obtained by identifying each class in
$\widehat K_\R$ with its highest weight.}\end{footnote} $\widehat K_\R$ to the group of roots of unity such that,
after a suitable rescaling of one of the hermitian forms,
\begin{equation}
\label{grvsur}
(v,v) = c\,\psi(j)(v,v)_{\fu_\R}\ \ \text{for $j\in \widehat K_\R$ and any $j$-isotopic vector $v\in V$.}
\end{equation}
Both $\,c\,$ and $\,\psi\,$ can be easily and explicitly described. In this identity $\,c\,\psi(j)= \pm 1\,$,\,\ of
course, unless $V$ contains no $j$-isotypic vectors $v\neq 0$. When $G_\R$ and $K_\R$ have unequal rank, the
relation between the two hermitian forms is only slightly less direct.

The close connection between the two hermitian forms means that the problem (\ref{subproblems}c) can be answered if
one has sufficiently precise information about the $\fu_\R$-invariant hermitian form.

Our conjecture is best stated in the context of $\mathcal D_\l$-modules, and we shall do so later in this note. But
for applications, the case of Harish Chandra modules is most important. For this reason, in the introduction we
shall concentrate on the Harish Chandra case.

Let $V$ be a Harish Chandra module with real infinitesimal character. We do not assume that $V$ is irreducible, but
$V$ should be ``functorially constructible". In particular, this includes -- but is not limited to -- so-called
standard Harish Chandra modules \cite{HMSWI}, their duals, and extensions of the former by the latter, as
considered by Beilinson-Bernstein in their proof of the Jantzen conjectures \cite{bb:1993}. We shall show that in
these cases $V$ carries two canonical, functorial filtrations: an increasing filtration by finite dimensional
$K$-invariant subspaces,
\begin{equation}
\label{hodgefilt1}
0 \, \subset \, F_a V \, \subset \, F_{a+1} V \, \subset \, \dots \, \subset \, F_pV \, \subset \, \dots \, \subset \, {\cup}_{a \leq p < \infty} \ F_pV \ = \ V\,,
\end{equation}
which we call the {\it Hodge filtration\/}, and a finite filtration by Harish Chandra submodules,
\begin{equation}
\label{weightfilt}
0 \, \subset \, W_b V \, \subset \, W_{b+1}V \, \subset \, \dots \, \subset \, W_kV \, \subset \, \dots \, \subset \, W_c V \ = \ V\,,
\end{equation}
the {\it weight filtration\/}. If $V$ is functorially constructible, then so are the $W_kV$ and the quotients
$W_kV/W_{k-1}V$. The Hodge filtration is a {\it good filtration\/}, in the sense that
\begin{equation}
\label{hodgefilt2}
\fg(F_pV)\subset F_{p+1}V  \ \ \text{for all}\,\ p\,,\,\ \ \text{and} \ \ F_pV +\, \fg(F_pV) = F_{p+1}V\ \ \text{for}\,\ p \gg 0 \,.
\end{equation}
Functorality means not only that the filtrations are preserved\begin{footnote}{Up to a shift of indices that can be
described explicitly.}\end{footnote} by morphisms of Harish Chandra modules, i.e., by simultaneously $U(\fg)\,$-
and $K$-invariant linear maps, provided they are functorially constructible in the same sense as the Harish Chandra
modules in question. The filtrations are induced by geometrically defined Hodge and weight filtrations of the
$\mathcal D_\l$-module realization of $V$, and therefore inherit all the functorial properties of these geometric
filtrations. In particular, like all filtrations in Hodge theory, both $F_{\cdot}V$ and $W_{\cdot}V$ are strictly
preserved by morphisms:
\begin{equation}
\label{strictness}
T(F_pV_1) = (TV_1)\cap F_pV_2 \ \ \ \text{and} \ \ \ T(W_kV_1) = (T\,V_1)\cap W_kV_2
\end{equation}
whenever $T: V_1 \to V_2$ is a functorially constructible morphism of Harish Chandra modules. Moreover,
\begin{equation}
\label{semisimple}
W_kV/W_{k-1}V\ \ \text{is completely reducible for each $k$.}
\end{equation}
In the case of standard modules, this latter assertion can be strengthened. Indeed, the weight filtration then
coincides with the Jantzen filtration \cite{bb:1993}, and hence with the socle filtration of $V$.

Still assuming that $V$ has a real infinitesimal character, we now also suppose that $V$ is irreducible. The
$\fu_\R$-invariant hermitian form $(\ ,\ )_{\fu_\R}$ on $V$ is determined only up to scaling. Vogan and his
coworkers show that it has the same sign on all the lowest $K$-types, and use this fact to fix a choice of sign. We
shall give a geometric construction of $(\ ,\ )_{\fu_\R}$, which also results in a preferred choice of sign. Our
choice, it turns out, agrees with that of Vogan et al. After these preparations, we can state:

\begin{conj}\label{conj1}
\ \ $(-1)^{p-a}(v,v)_{\fu_\R} > 0$\ \ \ {\rm if}\ \ \ $v\in F_pV \cap (F_{p-1}V)^\perp\,,\ \ v\neq 0\,$.
\end{conj}

Here $\,a\,$ denotes the lowest index in the Hodge filtration, as in (\ref{hodgefilt1}). The Hodge filtration is
typically very difficult to compute. On the other hand, it has excellent functorial properties. Our conjecture, if
proved, would help make the unitary dual accessible to geometric and functorial methods.

Vogan and his coworkers are aiming at an algorithm, one that could be implemented on a large scale computer, to
determine whether a given irreducible Harish Chandra module is unitarizable. In effect, they want to compute the
signature character of the hermitian form $(\ ,\ )_{\fu_\R}$, i.e., the sum
\begin{equation}
\label{signaturechar}
{\sum}_{j\in\widehat{K_\R}}\ (p_j - q_j)\,\phi_j\ ,
\end{equation}
with $\,\phi_j$=\,character of $\,j\in\widehat{K_\R}\,$, and $p_j,\,q_j$ denoting the multiplicity with which
$\,j\,$ occurs in $(\ ,\ )_{\fu_\R}$ with positive and negative sign, respectively. This depends on knowing the
change of the signature character of standard modules as the inducing parameter crosses the various reduction hyperplanes; they have a conjectured formula for this change \cite{vogan2009}. Our conjecture would transparently describe the change of
signature across reduction points, and would thereby imply their conjecture. We should emphasize that the Hodge
filtration is potentially a much finer invariant than the signature character. One might hope that our conjecture,
if proved, will lead to a much more efficient unitarizability algorithm. It would have other consequences as well.

At various periods during our collaboration on this project, we were guests of the MFO Oberwolfach, the MPI Bonn,
and the University of Essen. We thank all three institutions for their hospitality. We are also indebted to Dragan
Mili\v{c}i\'c for helpful discussions.\vspace{2pt}

\section{Invariant hermitian forms}\label{sec:hermitian forms}

Let us be more precise about the class of groups we are considering than we were in the
introduction\begin{footnote}{Our hypotheses on $G_\R$ are the most natural from an expository point of view, but
certainly not the most general under which our results are valid. To extend these results, along the lines sketched
in the appendix of \cite{HMSWI}, is mostly a technical exercise.}\end{footnote}. We start with $G$, a connected,
complex, reductive algebraic group, defined over $\R$\,.\, \ The group $G_\R$ whose unitary representa\-tions are
of interest to us can be any group between the group of all real points in $G$ and the identity component of the
group of real points. Then $G_\R$ has a finite number of connected components, and $\Ad g$, for $g\in G_\R$\,,\,\
defines an inner automorphisms of the Lie algebra of $G$; however $\Ad g$ may not be an inner automorphism of the
Lie algebra of $G_\R$.

We fix a maximal compact subgroup $K_\R \subset G_\R$, whose complexification we denote by $K$. The choice of
$K_\R$ is unique up to an inner automorphism, and is therefore not an essential choice. There exists a unique
compact real form $U_\R\subset G$ such that $U_\R \cap G_\R = U_\R \cap K = K_\R$. As always in this paper, we
denote the Lie algebras by the corresponding lower case German letters. The conjugate linear automorphisms
\begin{equation}
\label{cartandecomp1}
\sigma\,,\, \tau \ :\ \fg\, \to\, \fg\,,\ \ \ \ \left. \sigma_{\,} \right|_{\fg_\R} \, = \, 1\,,\ \ \ \ \left. \tau_{\,} \right|_{\fu_\R} \, = \, 1\,,
\end{equation}
commute, and
\begin{equation}
\label{cartandecomp2}
\theta \ = \ \sigma\,\tau\ = \ \tau\,\sigma
\end{equation}
is the Cartan involution \cite{helgason1962}. As usual, $\fp$ and $\fp_\R$ shall denote the $(-1)$-eigen\-spaces of
$\theta$ in $\fg$ and $\fg_\R$. Then
\begin{equation}
\label{cartandecomp3}
\fg \ = \ \fk \oplus \fp\,,\ \ \ \fg_\R \ = \ \fk_\R \oplus \fp_\R\,,\ \ \ \fu_\R \ = \ \fk_\R \oplus i\,\fp_\R\,,\ \ \
\end{equation}
are the Cartan decompositions of $\fg$, $\fg_\R$, and $\fu_\R$\,.

Any two Cartan subalgebras of $\fg$ are conjugate under an inner automorphism, and this automorphism becomes unique
if each of the two Cartans has been ``ordered"~-- i.e., equipped with a choice of positive root system -- and if,
moreover, the automorphism is required to preserve the order. By definition, the universal Cartan $\fh$ is a
representative of the unique isomorphism class of ordered Cartans, under inner automorphisms preserving the order.
The abstract Weyl group $W$ of $\fg$ acts on $\fh$, the $W$-action preserves the root system $\Phi(\fh)\subset
\fh^*$ as well as the weight lattice $\Lambda \subset \fh^*$, and $\Phi(\fh)$ contains a distinguished positive
root system $\Phi^+(\fh)$. Note that the real form
\begin{equation}
\label{realform1}
\fh_\R^* \ =_{\text{def}}\ \R\otimes_\Z \Lambda\ \subset \ \fh^*
\end{equation}
has nothing to do with the real forms $\fg_\R$ and $\fu_\R$ of $\fg$\,;\,\ in fact, when $\fh$ is identified with
the complexification of a Cartan subalgebra of $\fu_\R$, then $\fh_\R^*$ consists of purely imaginary elements.
Once and for all we fix a faithful algebraic representation of $G$. The trace form of that representation induces a
$W$-invariant inner product $(\ ,\ )$ on $\fh_\R^*$.

We write $\HC(\fg,K)$ for the category of Harish Chandra modules of the pair $(\fg,K)$, and $\HC(\fg,K)_\l$ for the
full subcategory of Harish Chandra modules with infinitesimal character $\chi_\l$\,,\,\ $\l\in\L$\,,\,\ in Harish
Chandra's notation. Then $\HC(\fg,K)_\l = \HC(\fg,K)_\mu$ if and only if $\l = w\mu$ for some $w\in W$. It has been
known for a long time that isomorphism classes of irreducible unitary representations of $G_\R$ correspond
bijectively to irreducible Harish Chandra modules with a positive definite $\fg_\R$-invariant hermitian form, again
up to isomorphism \cite{HC1}*{theorem 9}. That reduces the problem of describing the unitary dual $\widehat G_\R$
to the determination of all irreducible Harish Chandra modules with a positive definite hermitian form. We shall
say that
\begin{equation}
\label{realinfchar1}
\begin{gathered}
V \in \HC(\fg,K) \ \ \text{has real infinitesimal character}
\\
\text{if}\ \ V \in \HC(\fg,K)_\l\,,\ \ \text{with}\ \ \l\in\fh_\R^*\,.
\end{gathered}
\end{equation}
In addressing the unitarity problem for $G_\R$, one can restrict attention to irreducible Harish Chandra modules
with real infinitesimal character, without essential loss of generality. In fact, if we identify $\fh$ with the
complexified Lie algebra of a concrete, maximally split Cartan subgroup $H_\R\subset G_\R$, any irreducible
unitarizable Harish Chandra module $V\in \HC(\fg,K)_\l$ arises as the Harish Chandra module of a unitarily,
irreducibly induced representation $\operatorname{Ind}_{P_\R}^{G_\R}\tau$, and the Harish Chandra module of the
inducing representation $\tau$ is unitarizable, with real infinitesimal character; the Langlands component $M_\R$
of $P_\R$ coincides with the centralizer of $\Im \l$, and $\Im \l$ is the inducing parameter\begin{footnote}{Here
$\Im\l$ should be interpreted using the real structure (\ref{realform1}).}\end{footnote} \cite{knapp1986}*{theorem
16.10}.

Because of the reduction of the unitarity problem we just described, we may and shall impose the standing
hypothesis that
\begin{equation}
\label{realinfchar2}
\l\in \fh_\R^* \,,\, \ \text{and $\l$ is dominant,}
\end{equation}
in the sense that $(\alpha,\l) \geq 0$ for all $\alpha\in \Phi^+(\fh)$. As subscript in $\HC(\fg,K)_\l$, the
parameter $\l$ is only determined up to the action of $W$. Every $\l\in\fh_\R^*$ has a unique dominant
$W$-translate, so the dominance condition pins down $\l$ uniquely within its $W$-orbit. Initially the dominance
condition does not matter, but it will play a crucial role in the last section.

\begin{thm}\label{thm:URinvariantform}
{\rm (Vogan et al.\,\cite{vogan2009})}\,\ Suppose that $V\in \HC(\fg,K)_\l$ is irreducible, with $\l$ as in {\rm
(\ref{realinfchar2})}. Then $V$ carries a non-zero $\fu_\R$-invariant hermitian form $(\ ,\ )_{\fu_\R}$. It is
unique up to scaling, and can be rescaled so that it becomes positive definite on the lowest $K$-types in the sense
of Vogan \cite{vogan1979}.
\end{thm}

We shall construct the $\fu_\R$-invariant hermitian form geometrically in section \ref{sec:flag variety}. Vogan and
his collaborators not only established the existence of the $\fu_\R$-invariant hermitian form by algebraic methods,
and had the insight to relate it to the $\fg_\R$-invariant form when the latter exists. Their observation can be
stated as follows:

\begin{prop}\label{prop:urgr}
Let $V\in \HC(\fg,K)_\l$ be an irreducible Harish Chandra module as in theorem {\rm \ref{thm:URinvariantform}},
which in addition carries a non-zero $\fg_\R$-invariant hermitian form $(\ ,\ )$. When $(\ ,\ )$ is suitably
rescaled by a positive constant, there exists a $K$-invariant linear map $T: V \to V$ such that
\begin{enumerate}
\vspace{-2pt}
\item[{\rm i)}]\ \ $(u,v) \,=\, (Tu,v)_{\fu_R}$\,\ for all $\,u,v\in V$\,,
\item[{\rm ii)}]\ \ $T^2 \,=\, 1$\,,\,\ and
\item[{\rm iii)}]\ \ $T\,\zeta\,v \,=\, (\theta \zeta)\,(T\,v)$\,\ for all $\,v\in V\,,\,\zeta\in\fg$\,.
\end{enumerate}
\end{prop}

The linear map $T$ can also be described explicitly, in geometric terms. That will be done in section \ref{sec:flag
variety}. If $G$ and $K$  have the same rank, the relationship between the two hermitian forms has a simple,
completely explicit, combinatorial description; see corollary \ref{cor:urgr} below.

\begin{proof}
Let $V^c$ denote the complex conjugate of the space of $K$-finite vectors in the algebraic dual $V^*$ of $V$, or
equivalently, the space of conjugate linear maps from $V$ to $\C$ that vanish on a subspace of finite codimension.
The two hermitian forms induce two conjugate linear isomorphisms $i_{\fu_\R}, i_{\fg_\R} : V
\overset{{}_\sim}{\rightarrow} V^c$. Then $T = i_{\fu_\R}^{-1}\circ i_{\fg_\R}: V \overset{{}_\sim}{\rightarrow} V$
is a $\C$-linear isomorphism, and $T$ satisfies the condition i) by construction. Both hermitian forms are
$K_\R$-invariant, hence $T$ has this property as well. But the complexification $K$ meets every connected component
of $K_\R$, so the $\C$-linear map $T$ commutes with the action of all of $K$. On the infinitesimal level this
implies iii), at least for any $\zeta\in\fk$\,.\,\ Any $\zeta\in\fp$\,\ acts in a symmetric manner with respect to
$(\ ,\ )_{\fu_\R}$ and skew symmetrically with respect to $(\ ,\ )$; cf. (\ref{cartandecomp3}). Thus iii) also
applies to any $\zeta\in\fp_\R$\,, hence finally to any $\zeta\in\fg$\,.\,\ Using iii) twice, we find that $T^2$
commutes with the action of both $K$ and $\fg\,$~-- in other words, $T^2$ is an automorphism of the irreducible
Harish Chandra module $V$, and consequently $T^2 = c\,1$ with $c\in\C^*$. If we scale $T$ by $|c|^{-1/2}$ and $(\
,\ )$ by $|c|^{1/2}$, we can arrange that $|c|=1$ after rescaling, without destroying i) and iii). Since the lowest
$K$-types have multiplicity one, the $K$-invariant linear map $T$ must act as a multiple of the identity on each of
the minimal $K$-types, necessarily by a real multiple of the identity since $(v,v)_{\fu_R}$ and $(v,v)$ take only
real values. Changing the sign of $(\ ,\ )$ if necessary, we can arrange that $T^2$ acts by a positive constant, hence $c=1$, implying ii).
\end{proof}

Let $V^+$ and $V^-$ denote the $+1$ and $-1$ eigenspace of $T$, respectively. Then, as a formal consequence of the
proposition, \vspace{-2pt}
\begin{equation}
\begin{aligned}
\label{V+-}
\!\!\text{\rm a)}\ \ &V\ = V^+ \oplus V^-\,\ \ \text{($K$-invariant direct sum)}\,,\vspace{-2pt}
\\
\!\!\text{\rm b)}\ \ &\fp\,V^+ \subset V^-\,,\ \ \ \fp\,V^- \subset V^+\,,\vspace{-2pt}
\\
\!\!\text{\rm c)}\ \ &V^+\ \text{and}\ V^-\ \text{are orthogonal relative to both $(v,v)_{\fu_R}$ and $(v,v)$}\,,\vspace{-2pt}
\\
\!\!\text{\rm d)}\ \ &v\in V^\epsilon\, \Longrightarrow \ (v,v)_{\fu_R}=\epsilon\,(v,v)\,,\, \ \text{for both $\,\epsilon=+\,$ and $\,\epsilon=-$}\ .
\end{aligned}\vspace{-2pt}
\end{equation}
If $\rk G = \rk K$, the decomposition (\ref{V+-}\,a) can be described completely explicitly, as we shall see next.
In this situatation $G_\R$ contains a compact Cartan subgroup $T_\R$\,,\,\ which we may and shall assume lies in
$K_\R$\,.\,\ We temporarily identify $\fh$ with the complexification of the Lie algebra $\ft_\R$ of $T_\R$\,.\,\
Then $\ft_\R = i\fh_\R$\,,\,\ since weights take purely imaginary values on $\ft_\R$\,. One calls a root $\alpha\in
\Phi(\fh)\cong \Phi(\ft,\fg)\,$ $\ft_\R$-com\-pact or $\ft_\R$-noncompact depending on whether the $\alpha$-root
space $\fg^\alpha$ is contained in $\fk$ or $\fp$. These notions are preserved by the action of the ``real Weyl
group" $W(T_\R,G_\R) = N_{G_\R}(T_\R)/T_\R$, since $K_\R$ contains the normalizer $N_{G_\R}(T_\R)$.  It follows
that the notions of $\ft_\R$-compactness and $\ft_\R$-noncompactness do not depend on the particular choice of
identification $\fh \cong \ft = \C\otimes_\R \ft_\R$\,.\,\  Let $\Lambda_\Phi \subset \Lambda$ denote the lattice
spanned by the roots. Then
\begin{equation}\label{urgr1}
\psi : \Lambda_\Phi \to \{\pm 1\}\ ,\ \ \ \ \psi(\alpha)\ = \ \begin{cases}\ 1\,\ &\text{if $\alpha$ is $T_\R$-compact}\,,\\ \ -1\,\ &\text{if $\alpha$ is $T_\R$-noncompact}\,, \end{cases}
\end{equation}
is a well defined character since the sum of two roots is $T_\R$-compact if and only if the summands are both
$T_\R$-compact or both $T_\R$-noncompact. The Harish Chandra module $V$ breaks up as the direct sum of the various
$K$-isotypic subspaces,
\begin{equation}\label{K-isotypic1}
\begin{aligned}
V \, = \, {\oplus}_\mu\ V(\mu)\,,\  V(\mu)\,=\, \text{$K$-isotypic subspace of highest weight $\mu$\,,}
\\
\text{with \ \ $\mu \in \{\lambda \in \Lambda\,\mid\, (\alpha,\lambda)\geq 0$\ \ if $\alpha\in\Phi^+(\fh)$ is $T_\R$-compact.}\}
\end{aligned}
\end{equation}
We single out $\mu_0$, the highest weight of one of the lowest $K$-types. Then $V=U(\fg)V(\mu_0)$ by
irreducibility. Hence $V(\mu)$ can be nonzero only when $\mu-\mu_0\in\Lambda_\Phi$. We know that $T$ acts on
$V(\mu_0)$ either as the identity or as $-1$. Replacing $(\ ,\ )$ by its negative, if necessary, we may as well
suppose that $T=1$ on $V(\mu_0)$. Also, if $\zeta$ lies in the $\alpha$-root space $\fg^\alpha$, $T\zeta T^{-1}=
(-1)^{\psi(\alpha)}\zeta$\,\ by assertion iii) of proposition \ref{prop:urgr}. This implies:

\begin{cor}\label{cor:urgr}
In the equal rank case, if $V(\mu_0) \subset V^+$ as may be assumed, the isotypic subspace $V(\mu)$ lies in $V^+$
or $V^-$ depending on wether $\psi(\mu-\mu_0)$ equals $\,+1$ or $\,-1\,$.
\end{cor}

In particular, $V^+$ and $V^-$ have no $K$-types in common. That remains true even in the unequal rank case,
provided one replaces $G_\R$ with the appropriate extension of $G_\R$ by the two element group $\{1,T\}$.

Back in the equal rank case, one can extend the character $\psi$ of (\ref{urgr1}) to a homomorphism from the full
weight lattice to the group of roots of unity,
\begin{equation}\label{urgr2}
\psi \, : \,\Lambda \, \to \, \{\,z\in \C\,\mid \, z^N=1\,\,\ \text{for some $N\in\Z_{>0}$\,}\}\,.
\end{equation}
In terms of this notation, with $c=\psi(-\mu_0)$, the relation between the two hermitian forms can be restated as
in (\ref{grvsur}), in the introduction.\vspace{2pt}

\section{Brief review of the theory of mixed Hodge modules}\label{sec:MHM_summary}

In the sequence of papers \cites{saito:rims1989,saito:compositio1990,saito:rims1990,saito:rims1991} Morihiko Saito
developed his theory of mixed Hodge modules. We shall need a slightly extended version of the theory: Saito
requires the existence of a rational structure, which implies that the local monodromy transformations for the
underlying local systems have roots of unity as eigenvalues. For our applications we need to relax this hypothesis;
all we can assume is that the eigenvalues have absolute value one, and this follows from the presence of a
polarization. On the other hand, we shall only consider complex algebraically constructible sheaves and,
correspondingly, $\cD$-modules\begin{footnote}{For us, ``$\cD$-module" shall always mean left $\cD$-module. That is
not the case in Saito's papers.}\end{footnote} in the complex algebraic setting. Not all of the theory of mixed
Hodge modules depend on this latter restriction, but it does simplify various statements.

We begin with the notions of complex Hodge structure and complex mixed Hodge structure. Let $H$ be a finite
dimensional complex vector space. A complex Hodge structure of weight $n$ on $H$ consists of the datum of a direct
sum decomposition
\begin{equation}
\label{hodge1}
H \ = \ {\oplus}_{p+q=n}\, H^{p,q}\,.
\end{equation}
The cohomology group in degree $n$ of a compact K\"ahler manifold carries a natural Hodge structure of weight $n$.
In that case, $n\geq 0$ and the summation extends only over non-negative integers $p$, $q$. In general, $p$ and $q$
are allowed to be arbitrary integers adding up to $n$, which need not be positive, either. The Hodge structure
(\ref{hodge1}) determines, and is determined by, the two increasing filtrations $F_\cdot \,H$ and $\overline
F_{\!\cdot}\, H$,
\begin{equation}
\label{hodge2}
F_p \,H \ = \ {\oplus}_{r\geq -p}\, H^{r,n-r},\ \ \ \overline F_{\!q} \,H\ = \ {\oplus}_{s\geq -q}\,\, \overline H^{n-s,s} \,;
\end{equation}
indeed, the identity
\begin{equation}
\label{hodge3}
H^{p,q} \ = \ F_{-p}\,H \cap \overline F_{\!-q}\, H
\end{equation}
recovers the summands $H^{p,q}$ from the two filtrations. One calls $F_\cdot \,H$ the Hodge filtration, and
$\overline F_{\!\cdot}\, H$ is the conjugate Hodge filtration. Traditionally one uses the decreasing filtrations
$\{F^p H\} = \{F_{\!-p} \,H\}$ and $\{\overline F^{\,q} H\} = \{\overline F_{\!-q}\, H\}$ instead. However, in the
context of $\mathcal D$-modules, which crucially enter the theory of mixed Hodge modules, increasing filtrations
are more natural. For this reason we work with increasing filtrations right from the beginning. Also, in Hodge
theory one customarily assumes that the underlying vector space $H$ is defined over $\Q$, and that $H^{q,p}$ is the
complex conjugate of $H^{p,q}$. In that case, the complex conjugate vector space $\overline H$ coincides with $H$,
and $\overline F_{\!\cdot}\, H$ is the complex conjugate of the filtration $F_\cdot \,H$, as the notation suggests.
We do not require the existence of a rational structure, or even a real structure, and the notation $\overline
F_{\!\cdot}\, H$ has merely symbolic significance. Of course one can extend a complex Hodge structure
(\ref{hodge2}) on $H$ to one on $H \oplus \overline H$, which has at least a real structure, by formally taking the
direct sum of $H= \oplus\, H^{p,q}$ and $\overline H = \oplus\,{\overline H}^{\,p,q}$,\,\ with ${\overline
H}^{\,p,q}=\overline {H^{\,q,p}}$.

One can take not only the direct sum of two complex Hodge structures, provided they have the same weight, but also
the tensor product of any two, of possible different weights, as well as symmetric and tensor powers and the dual
of a single complex Hodge structure. The weights behave in the obvious way. A morphism of complex Hodge structures
is a linear map between the underlying vector spaces which preserves the Hodge types, or equivalently, the two
filtrations. By definition, a polarization of the Hodge structure (\ref{hodge1}) is a -- necessarily nondegenerate
-- hermitian pairing $Q : H \times \overline H \to \C$ such that
\begin{equation}
\begin{aligned}
\label{hodge4}
&\text{a)}\,\ Q(H^{p,q},H^{r,s}) = 0\,\ \text{if}\,\ (r,s) \neq (p,q)\,,\,\ \text{and}
\\
&\text{b)}\,\ (-1)^p\,(2\pi i)^n\, Q(v,\overline v)>0\,\ \text{for all}\,\ v\in H^{p,q}, \,\ v\neq 0\,.
\end{aligned}
\end{equation}
We call the datum $(H,\,F_\cdot \,H,\, Q\,)$ a polarized complex Hodge structure of weight $n\,$. Note that
$H^{p,q} = F_{-p}H \cap (F_{\!-p-1}\,H)^\perp\!$,\,\ so that
\begin{equation}
\label{hodge5}
\overline F_{q}\,H\ = \ {\oplus}_{s\geq -q}\, H^{n-s,s} \ = \ {\oplus}_{s\geq -q}\, \( F_{s-n}\,H \cap (F_{s-n-1}\,H)^\perp\)\,.
\end{equation}
The conjugate Hodge filtration $\overline F_{\!\cdot}\, H$ can therefore be reconstructed from the Hodge
filtration, in conjunction with the polarization.

A complex mixed Hodge structure on the finite dimensional complex vector space $H$ consists of three increasing
filtrations $F_\cdot \,H$, $\overline F_{\!\cdot}\, H$, $W_\cdot \,H$, such that for all integers $k$,
\begin{equation}
\begin{aligned}
\label{mhodge1}
&\text{the induced filtrations}\,\ F_p(W_k H)=F_p \,H\cap W_kH\,,
\\
&\ \ \ \ \ \ \overline F_q(W_k H)=\overline F_q \,H\cap W_kH\,,\,\ \text{determine a complex}
\\
&\ \ \ \ \ \ \ \ \ \ \ \ \text{Hodge structure of weight $k$ on}\,\ W_kH / W_{k-1}H\,.
\end{aligned}
\end{equation}
One calls $W_\cdot \,H$ the weight filtration. If each of the quotients $W_kH / W_{k-1}H$ comes equipped with a
specific polarization, one calls that a graded polarized complex mixed Hodge structure. As in the case of Hodge
structures of pure weight, one customarily assumes that $H$ is defined over $\Q$, that $\overline F_{\!\cdot}\, H$
is the complex conjugate of $F_\cdot \,H$, and in addition, that the weight filtration is defined over $\Q$. As
before, one can embed any complex mixed Hodge structure into one with underlying real structure on $H \oplus
\overline H$. The linear algebra of mixed Hodge structures -- which is not entirely trivial -- depends not on the
rational structure, but only on the real structure. All conclusions therefore apply in the complex case. In
particular morphisms of complex mixed Hodge structures -- i.e., linear maps between the underlying vector spaces
that preserve the three types of filtration -- preserve the filtrations strictly:
\begin{equation}
\begin{aligned}
\label{mhsstrictness}
&T(F_pH_1) = (TH_1)\cap F_pH_2 \,,\ \ T(\overline F_{\!q} H_1) = (TH_1)\cap \overline F_{\!q}H_2\,,
\\
&\qquad\qquad\ \ \ \ \text{and} \ \ \ T(W_kV_1) = (T\,V_1)\cap W_kV_2
\end{aligned}
\end{equation}
whenever $T: H_1 \to H_2$ is a morphism of complex mixed Hodge structures.

A polarized complex variation of Hodge structure of weight $n$\,, on a complex smooth quasi-projective variety $Z$,
is specified by the datum of a vector bundle ${\bH} \to Z$, a flat connection $\nabla$ on $\bH$, a flat hermitian
form $Q$ on the fibers of $\bH$, and an increasing filtration of the sheaf $\O_Z({\bH})$ by locally free sheaves of
$\O_Z$-modules
\begin{equation}
\label{vhs1}
0 \subset F_a \O_Z({\bH}) \subset F_{a+1}\O_Z({\bH}) \subset \dots \subset F_p\O_Z({\bH}) \subset \dots \subset \O_Z(\bH)\,,
\end{equation}
subject to the following two conditions:\addtocounter{equation}{1}
\begin{equation*}
\label{vhs2*}
\nabla F_p\O_Z({\bH}) \ \subset \ F_{p+1}\O_Z({\bH})\ \ \text{for all $p$}\,,\tag*{(\theequation a)}
\end{equation*}
and\bigskip
\begin{equation*}
\begin{aligned}
\label{vhs3*}
&\text{for each $z\in Z$,}\,\ \big(\,{\bH}_z,\,\left(F_\cdot \O_Z({\bH})\right)/\left(\mathcal I_z F_\cdot \O_Z({\bH})\right),\, Q_z\,\big)\,\ \text{is a}
\\
&\qquad\ \text{polarized Hodge structure on ${\bH}_z$, of weight $n$}\,;
\end{aligned} \tag*{(\theequation b)}
\end{equation*}
here $\mathcal I_z$ denotes the ideal sheaf of $\{z\} \subset Z$,\,\ ${\bH}_z$ the fiber of $\bH$ at $z$, and $Q_z$
the restriction of $Q$ to ${\bH}_z$. One refers to the condition \ref{vhs2*} as Griffiths transversality. In
\ref{vhs3*} the conjugate Hodge filtration is not specified; as mentioned earlier, in the presence of a
polarization the Hodge filtration completely determines the conjugate Hodge filtration.

As usual, $ \cD_Z$ shall denote the sheaf of linear differential operators on the smooth quasi-projective variety
$Z$, with algebraic coefficients. It is a sheaf of algebras over $\O_Z$, and is naturally filtered by the degree of
differential operators. We let $\Mod(\cD_Z)_{\text{coh}}$ denote the category of (sheaves of) $\cD_Z$-modules which
are coherent over $\cD_Z$. A good filtration of $\cM\in \Mod(\cD_Z)_{\text{coh}}$ is an increasing filtration by
$\O_Z$-submodules $F_p\cM \subset \cM$, coherent over $\O_Z$, such that
\begin{equation}
\label{goodfilt1}
(\cD_Z)_k F_p\cM \subseteq F_{p+k}\cM \, \ \text{for all $k,\,p$, with equality holding for}\ p \gg 0 \,;
\end{equation}
$(\cD_Z)_k=F_k\cD_Z$ refers to the filtrants of $\cD_Z$, of course. Good filtrations exist, at least locally. The
associated graded $\gr_ \cdot\cM$ with respect to a good filtration $F_\cdot\cM$ can be regarded as a coherent
sheaf of $\O_{T^*Z}$-modules\begin{footnote}{more precisely, $\gr\cM = \pi_*\cF$ is the sheaf direct image, under
the natural projection $\pi : T^*Z \to Z$, of a graded coherent $\O_{T^*Z}$-module $\cF$ on $T^*Z$, and $\gr_\cdot
\cD_Z = \pi_*\O_{T^*Z}$.}\end{footnote} on the cotangent bundle $T^*Z$\,. As is easy to see, the support of
$\gr_\cdot \cM$ in $T^*Z$~-- the so-called singular support of $\cM$~-- does not depend on the particular choice of
a good filtration. We recall that $\cM$ is said to be holonomic if its singular support has the minimal possible
dimension, or equivalently, if the singular support is Lagrangian. Alternatively and equivalently one can
characterize the property of being holonomic by the cohomological condition
\begin{equation}
\label{holonomic1}
\cExt^k_{\cD_Z}(\cM,\cD_Z)\ = \ \rm 0\ \ \ \text{if}\ \ k\neq \dim Z\,.
\end{equation}
In other words, $\cM$ is holonomic if and only if it is Cohen-Macaulay.

Mixed Hodge modules involve filtered $\cD$-modules that are are regular holonomic. Regular holonomic $\cD$-modules
are most easily defined in a functorial fashion: one can describe them as constituting the smallest full
subcategory of the category of holonomic modules invariant under the usual operations on $\cD$-modules, invariant
under extensions, and containing all $\cD$-modules generated by flat connections with regular singularities. It is
not obvious that such a category exists. That can be established by characterizing regular holonomic modules
directly, for example by reduction to the case of a one dimensional base manifold $Z$. It is then a non-trivial
matter to show that the the resulting category is indeed the smallest full subcategory of the category of holonomic
modules possessing the invariance properties spelled out before. Of course, ``invariance under the usual
operations" requires passage to the derived category, since some of them only exist on the derived level. The
definition of the category of mixed Hodge modules is entirely analogous, again with a functorial working
definition, and a direct characterization involving a reduction to the one dimensional case.

Let us write $\Mod(\cD_Z)_{\text{rh}}$ for the category of regular holonomic $\cD_Z$-modules, and
$\Mod(\cD_Z)_{\text{frh}}$ for the category of all $\cM \in \Mod(\cD_Z)_{\text{rh}}$ which come equipped with a
specific good filtration. Forgetting the filtration, one can regard each $\cM \in\Mod(\cD_Z)_{\text{frh}}$ as
defining an object in $\Mod(\cD_Z)_{\text{rh}}$. Morphisms in $\Mod(\cD_Z)_{\text{frh}}$ are required to preserve
the filtration, of course; a morphism $T: \cM \to \cN$ is said to be strict if $T(F_k\cM) = T(\cM)\cap F_k\cN$ for
all $k$. It was mentioned already that $\Mod(\cD_Z)_{\text{rh}}$ is stable under all the usual operations and
constructions. On the other hand, $\Mod(\cD_Z)_{\text{frh}}$ has very limited functorial properties in general. One
way to view the mixed Hodge modules is, roughly speaking, as specifying a subcategory of $\Mod(\cD_Z)_{\text{frh}}$
with funtorial properties as extensive as those of $\Mod(\cD_Z)_{\text{rh}}$.

For our purposes, the abelian category of complex mixed Hodge modules, $\CMHM(Z)$, is particularly relevant. Its
objects are pairs $(\cM,W_\cdot\cM)$, consisting of some $\cM\in\Mod(\cD_Z)_{\text{frh}}$ and a finite increasing
filtration $W_\cdot\cM$ of $\cM$ by submodules $\cN\in\Mod(\cD_Z)_{\text{frh}}$ of $\cM$, each equipped with the
good filtration induced by $\cM$. Not every such pair describes an object in $\CMHM(Z)$, of course. For reasons
that shall become apparent, we call $W_\cdot\cM$ the weight filtration and the intrinsic good filtration of $\cM$
the Hodge filtration. A morphism in $\CMHM(Z)$ is a morphism of the underlying $\cD_Z$-modules which preserves both
filtrations. At first glance this appears to be a weak requirement, but it turns out that any such morphism is
strict with respect to both filtrations. An object $(\cM,W_\cdot\cM)$ of $\CMHM(Z)$ is said to be pure of weight
$k$ if $W_k\cM = \cM$ and $W_{k-1}\cM = 0$. The full subcategory of pure objects is $\CHM(Z)$, the category of
complex mixed Hodge modules on $Z$. We denote the bounded derived categories of, respectively, $\CMHM(Z)$ and
$\Mod(\cD_Z)_{\text{rh}}$ by $D^b(\CMHM(Z))$ and $D^b(\Mod(\cD_Z)_{\text{rh}})$. We extend the notion of weights to
$D^b(\CMHM(Z))$ by the convention that
\begin{equation}
\label{weight_of_cohomology}
\begin{aligned}
&\text{$\cM\,$ has weight $\,\leq k\,$ if $\,H^\ell(\cM)\,$ has weight $\,\leq k+\ell\,$ for all $\ell$\,,\,\ and}
\\
&\text{$\cM\,$ has weight $\,\geq k\,$ if $H^\ell(\cM)\,$ has weight $\,\geq k+\ell\,$ for all $\ell$}\,.
\end{aligned}
\end{equation}
Thus, if $\cM\in D^b(\CMHM(Z))$ is pure of weight $k$, then $H^\ell(\cM)$ is pure of weight $k+\ell$\,.\,\ The
passage to the underlying filtered $\cD_Z$-module, followed by forgetting the filtration, defines a forgetful
functor from $\CMHM(Z)$ to $\Mod(\cD_Z)_{\text{rh}}$, and also on the level of the derived categories. We shall say
that a functor on $\Mod(\cD_\centerdot\,)_{\text{rh}}$ lifts to $\CMHM({}^{{}_{\displaystyle\centerdot}})$ if there
exists a commutative square incorporating the original functor and the two forgetful functors. We use the same
terminology for the derived categories, of course.

We can now state the properties which characterize $\CMHM(Z)$, as the smallest category satisfying them all. First
of all,\addtocounter{equation}{1}
\begin{equation*}
\label{defcr*}
\text{every $\cM\in\CHM(Z)$ is completely reducible},\tag*{(\theequation a)}
\end{equation*}
and for every $\cN\in\CMHM(Z)$ and every $k$,
\begin{equation*}
\label{defgr*}
\text{$\gr_{W,k}\cN\in\Mod(\cD_Z)_{\text{frh}}$ exists as object in $\CHM(Z)$, of weight $k$}.\tag*{(\theequation b)}
\end{equation*}
If $f:Z \to X$ is an algebraic map between smooth complex quasi-projective varieties,
\begin{equation*}
\label{deffunct*}
\begin{aligned}
&\text{the $\cD$-module functors $f_+, f_! : D^b(\Mod(\cD_Z)_{\text{rh}})\to D^b(\Mod(\cD_X)_{\text{rh}})$}
\\
&\ \ \ \ \text{and $f^+\!, f^! : D^b(\Mod(\cD_X)_{\text{rh}})\to D^b(\Mod(\cD_Z)_{\text{rh}})$ lift}
\\
&\ \ \ \ \ \ \ \ \text{to functors $f_*, f_! : D^b(\CMHM(Z))\to D^b(\CMHM(X))$},
\\
&\ \ \ \ \ \ \ \ \ \ \ \ \text{respectively $f^*, f^! : D^b(\CMHM(X))\to D^b(\CMHM(Z))$}.
\end{aligned}\tag*{(\theequation c)}
\end{equation*}
Because of the Cohen-Macaulay property (\ref{holonomic1}), the definition of the duality functor $\D$ does not
require passage to the derived category:
\begin{equation*}
\label{defduality*}
\begin{aligned}
&\text{\,\,i)\ \ the functor $\D$ lifts from $\Mod(\cD_Z)_{\text{rh}}$ to $\CMHM(Z)$\,, and}
\\
&\text{ii)\ \ if $\cN\in \CHM(Z)$ has weight $k$,\,\ $\D\cN$ is pure of weight $-k$\,.}
\end{aligned}\tag*{(\theequation d)}
\end{equation*}
Recall that the $\cD$-module dual $\D\cM$ of any $\cM \in\Mod(\cD_Z)_{\text{rh}}$ can be identified with
$\cExt^{\,\dim Z}_{\cD_Z}(\cM,\cD\otimes ({\rm \Omega}^{\,\dim Z})^{-\rm 1})$.  For singleton spaces,
\begin{equation*}
\label{singleton*}
\begin{aligned}
&\text{$\CMHM(\{\text{pt}\})$ is isomomorphic to the category}
\\
&\qquad\ \ \text{of complex mixed Hodge structures.}
\end{aligned}\tag*{(\theequation e)}
\end{equation*}
Complex mixed Hodge modules involve two filtrations, the Hodge and weight fitrations. Complex mixed Hodge
structures -- in particular, $\CMHM(\{\text{pt}\})$ -- involve the conjugate Hodge filtration in addition. This
apparent inconsistency will be resolved in the appendix. Finally,
\begin{equation*}
\label{defchm*}
\begin{aligned}
\text{every polarized variation of Hodge structure of weight $n$}
\\
\text{defines an object in $\CHM(Z)$, of weight $n+\dim Z$}.
\end{aligned}\tag*{(\theequation f)}
\end{equation*}
Of course this forces us define the weight filtration on $\O_Z({\bH})$ by decreeing it to be pure of weight $n +
\dim X$. The Hodge filtration of the resulting object in $\CHM(Z)$ coincides with the Hodge filtration of the
original variation of Hodge structure.

Let us mention two important properties of $\CMHM(Z)$ that follow from its definition. Recall the convention
(\ref{weight_of_cohomology}). If $f:Z\to X$ is an algebraic map,
\begin{equation}
\label{lowering_and_raising_weights}
\text{the functors $f^*$ and $f_!$ lower weights, and $f_*$ and $f^!$ raise them;}
\end{equation}
here ``lower" and ``raise" are to be taken in the sense of non-increasing and non-deceasing, respectively. Then
(\ref{lowering_and_raising_weights}) has clear meaning when the functors are applied to a pure object. In general,
the meaning is clarified by the strict\-ness property
\begin{equation}
\label{mhmstrictness}
\begin{aligned}
\text{every morphism of complex mixed Hodge modules preserves both}
\\
\text{the Hodge filtration and the weight filtration strictly.}
\end{aligned}
\end{equation}
In particular, there exist no nontrivial morphisms in $\CHM(Z)$ between objects of different weights.

Saito's categories of mixed Hodge modules $\MHM(Z)$ and of Hodge modules $\HM(Z)$ carry also a rational structure.
This rational structure, for an object $(\cM,W_\cdot\cM)$ of $\CMHM(Z)$, consist of a perverse sheaf of $\Q$-vector
spaces $\mathcal S$ and a concrete isomorphism between $\operatorname{DR}(\cM)$~-- the de Rham functor, applied to
$\cM\in \Mod(\cD_Z)_{\text{rh}}$~-- and the complexification $\C\otimes_\Q\mathcal S$ of $\mathcal S$. In effect,
Saito establishes the properties listed above in the presence of rational structures. In the appendix we shall
sketch the modifications of his arguments that are necessary to prove these properties also in the absence of
rational structures.

The category $\Mod(\cD_Z)_{\text{frh}}$ of filtered regular holonomic $\cD_Z$ was studied even before Saito's
theory of mixed Hodge modules \cite{laumon1983}. This category behaves well under the direct image under projective
morphisms, and inverse image under smooth morphisms. In those cases, the effect of the induced morphisms in the
category $\CMHM(Z)$ on the Hodge filtration is consistent with its effect in the category
$\Mod(\cD_Z)_{\text{frh}}$. An algebraic map $f: Z \to X$ between smooth complex quasi-projective varieties can be
decomposed into a closed embedding via its graph, followed by the projection $Z \times X \to X$. Thus, to describe
the functorial behavior of $\Mod(\cD_Z)_{\text{frh}}$ under projective direct image, it suffices to treat the two
special cases of closed smooth embedding and projection in a product with projective fiber. These two cases are
easy to describe explicitly, and they are relevant for our application.

First we consider a smooth closed embedding $i : Z \to X$. We can choose local algebraic coordinates $z_1, \dots,
z_m, z_{m+1},\dots,z_{m+n}$\,,\,\ such that $Z\subset X$ is cut out by the equations $z_1 = z_2 = \dots = z_m = 0$.
The direct image $i_+\cM$ of any $\cM\in \Mod(\cD_Z)_{\text{frh}}$ can be identified as $\cD_X$-module, locally in
terms of the chosen coordinates, with
\begin{equation}
\label{i*frh1}
i_+\cM \ \cong \ \cM[\textstyle\frac{\partial\ }{\partial{z_1}},\frac{\partial\ }{\partial{z_2}},\dots, \frac{\partial\ }{\partial{z_m}}]\,,
\end{equation}
i.e., polynomials in the normal derivatives $\frac{\partial\ }{\partial{z_j}}$\,,\,\ $1\leq j \leq m$\,,\,\ having
local sections of $\cM$ as coefficients. The filtration of $i_+\cM$ can then described in terms of the filtration
of $\cM$ by the identity
\begin{equation}
\label{i*frh2}
F_p\,i_+\cM \ \cong \ {\sum}_{q+|\alpha|+ m\leq p}\ (F_q \, \cM)\, \textstyle(\frac{\partial\ }{\partial{z_1}})^{\alpha_1}(\frac{\partial\ }{\partial{z_2}})^{\alpha_2}\dots(\frac{\partial\ }{\partial{z_m}})^{\alpha_m}\,,
\end{equation}
with $|\alpha|=\alpha_1+\dots+\alpha_m$\,.\,\ As was mentioned already, this describes the effect on the Hodge
filtration when $\cM$ arises as the underlying filtered $\cD_Z$-module of a object in $\CMHM(Z)$. In that
situation,
\begin{equation}
\label{i*frh3}
W_k\,i_+\cM \ \cong \ (W_k \, \cM)[\textstyle\frac{\partial\ }{\partial{z_1}},\frac{\partial\ }{\partial{z_2}},\dots, \frac{\partial\ }{\partial{z_m}}]\,,
\end{equation}
describes the weight filtration. Globally the passage from $\cM$ to $i_+\cM$ involves a twist by the top exterior
power of $T_Z^*X$, the dual normal bundle, but that does not affect the descriptions (\rangeref{i*frh2}{i*frh3}) of
the Hodge and weight filtrations.

To describe the effect of a projection $\pi : X \times Z \to Z$ with projective fiber $X$, we consider a filtered
module $\cM \in \Mod(\cD_{X\times Z})_{\text{frh}}$. Recall that if one disregards the filtration,
\begin{equation}
\label{pi+frh1}
\pi_+\cM \ = \ R\pi_*(\cM \otimes \Omega_X^{\,\cdot}[\dim X])\,,
\end{equation}
as an object of the derived category $D^b(\Mod(\cD_{X\times Z})_{\text{rh}})$. Here $(\Omega_X^{\,\cdot}[n])^p=
\Omega_X^{n+p}$, since the the index $p$ occurs as a superscript, in contrast to $(F_\cdot\cM[m])_p = F_{p-m}\,\cM$
in the case of a subscript; this is consistent with the convention of ``lowering the index" $F^p \leftrightarrow
F_{-p}$.  The shift by $n$ has the effect of putting the complex in degrees\,\ $-n,1-n,\dots,-1,0\,$.\,\ The
filtration
\begin{equation}
\label{pi+frh2}
F_p\pi_+\cM \, = \, R\pi_*(F_p\cM \to F_{p+1}\cM \otimes\Omega_X^1 \to \dots \to F_{p+n}\cM\otimes\Omega_X^n)[n]\,,
\end{equation}
with $n=\dim X$, specifies $\pi_+\cM$ as object in $D^b(\Mod(\cD_{X\times Z})_{\text{frh}})$. When $\cM$ underlies
an object in $\CMHM(X\times Z)$,
\begin{equation}
\label{pi+frh3}
\cM \ \ \text{has weight}\ \, k\ \ \ \Longrightarrow \ \ \ \pi_+\cM \ \ \text{has weight}\ \, k\,.
\end{equation}
In this connection one should recall (\ref{weight_of_cohomology}).

To understand the behavior of the Hodge filtration under general direct images one needs to treat also the case of
an open embedding. That is far more delicate and can only be done in the context of mixed Hodge modules, not for
general filtered $\cD$-modules; see the final remarks in section \ref{sec:flag variety} below.\vspace{2pt}

\section{The Beilinson-Bernstein construction}

We shall apply the theory of complex mixed Hodge modules to $U(\fg)$-modules with real infinitesimal character.
This involves Beilinson-Bernstein's notion of localization of $U(\fg)$-modules on the flag variety $X$, which we
shall review in this section. For the present purposes our ``standing hypothesis" (\ref{realinfchar2}) is
irrelevant, but we shall re-impose it after this section.

Recall that $X$, the variety of Borel subalgebras of $\fg$, is a smooth, complex projective variety, equipped with
a transitive algebraic action of $G$. If $\fb\subset \fg$ is a Borel subalgebra with unipotent radical $\fn$, the
universal Cartan algebra $\fh$ can be uniquely identified with any particular concrete subalgebra of $\fb$ so that
$\fn$ becomes the direct sum of the root spaces corresponding to the negative roots,
\begin{equation}
\label{hbn}
\fb \ = \ \fh \oplus \fn\,,\ \ \ \ \fn\ = \ {\oplus}_{\a\in\Phi^+}\,\fg^{-\a}\ .
\end{equation}
Essentially by definition, any $\l$ in the weight lattice $\,\L\subset \fh_\R^*$ lifts to an algebraic character
$e^\l$ of the Borel subgroup $B\subset G$ with Lie algebra $\fb$. There exists a unique algebraic, $G$-equivariant
line bundle $\cL_\l \to X$ on whose fiber over $\fb$ the isotropy group $B$ acts via $e^\l$, and every
$G$-equivariant algebraic line bundle on $X$ arises in this manner,
\begin{equation}
\label{Llambbda}
\L \ \cong \ \text{group of $G$-equivariant line bundles}\,,\ \ \ \l \ \mapsto \ \cL_\l\ .
\end{equation}
The line bundle $\cL_\l$ is ample in the sense of algebraic geometry precisely when $\l$ is dominant and regular,
i.e., when $(\a,\l)>0$ for all $\a\in\Phi^+$.

It will be convenient to assume temporarily that the algebraic group $G$ is simply connected; this hypothesis will
be relevant only in the next two paragraphs. One defines
\begin{equation}
\label{rhodef}
\rho\ \ = \ \ {\textstyle\frac 12}\ {\sum}_{\a\in\Phi^+}\ \a\ .
\end{equation}
Then $2\rho$ is a sum of roots, hence an element of the weight lattice. The fiber at $\fb$ of the tangent bundle of
$X$ is naturally isomorphic to $\fg/\fb$\,, hence
\begin{equation}
\label{ltworho}
\cL_{-2\rho}\ \ \cong \ \ \wedge^nT^*X\ \ = \ \ \text{canonical bundle on $X$}\ \ \ \ \ \text{(\,$n=\dim X$\,)}.
\end{equation}
According to a standard fact on algebraic groups,
\begin{equation}
\label{rhosc}
\text{$G$\ \ simply connected} \ \ \ \Longrightarrow \ \ \ \rho \in\Lambda\,,
\end{equation}
in which case the canonical bundle $\cL_{-2\rho}$ has a well defined square root $\cL_{-\rho}$.

We define $\cD_\l = \O(\cL_{\l-\rho})\otimes_{\O_X}\!\cD_X\otimes_{\O_X}\!\O(\cL_{\rho-\l})$. It is a
$G$-equivariant sheaf of associative algebras over $\O_X$, locally isomorphic to $\cD_X$, the sheaf of linear
differential operators on $X$ with algebraic coefficients. By definition, it acts on sections of $\cL_{\l-\rho}$\,.
The Lie algebra $\fg$ also acts, by infinitesimal translation, on sections of $\cL_{\l-\rho}$\,. Thus $\fg$ embeds
into the space of global sections of $\cD_\l$\,, $\fg \, \hookrightarrow\,\G\, \cD_\l\,$, and this inclusion
induces a morphism of associative algebras
\begin{equation}
\label{Llambda1}
U(\fg) \ \ \longrightarrow\ \ \G\, \cD_\l
\end{equation}
which is compatible with the degree filtrations on both sides. One can show easily that the center $Z(\fg)$ of
$U(\fg)$ acts on $\O(\cL_{\l-\rho})$ via the character $\chi_\l$\,, as defined in section \ref{sec:hermitian
forms}. It follows that the homomorphism (\ref{Llambda1}) drops from $U(\fg)$ to the quotient
\begin{equation}
\label{Ulambda}
U_\l(\fg) \ \ =_{\text{def}} \ \ U(\fg)/\Ker\{\chi_\l : Z(\fg) \to Z(\fg)\}\cdot U(\fg)\ ,
\end{equation}
as a homomorphism
\begin{equation}
\label{Llambda2}
U_\l(\fg) \ \ \longrightarrow\ \ \G\, \cD_\l
\end{equation}
which again is compatible with the degree filtrations.

Like any algebraic line bundle, the $G$-equivariant line bundle $\O(\cL_{\l-\rho})$ has local trivializations. Over
the intersection of any two trivializing neighborhoods, these are related by a nowhere vanishing, algebraic
transition function. Any local trivialization of $\O(\cL_{\l-\rho})$ simultaneously trivializes the sheaf of
algebras $\cD_\l$, and local trivializations of the latter are related via conjugation by the transition functions
for $\O(\cL_{\l-\rho})$. Conjugation of a differential operator by a nowhere vanishing function $f$ involves taking
logarithmic derivatives of $f$. Thus conjugation of differential operators not only by $f$ has meaning, but even
conjugation by any complex power $f^a$. Since $\fh^* = \C\otimes_\Z\L$, one can use this reasoning to show that
\begin{equation}
\label{Dlambda}
\cD_\l\ \ \text{is well defined for any $\,\l\in \fh^*$\,},
\end{equation}
as infinitesimally $\fg$-equivariant, filtered, associative sheaf of algebras over $\O_X$\,,\,\ a so-called
``twisted sheaf of differential operators". Unless $\l\in\L$, $\cD_\l$ cannot be realized as the sheaf of linear
differential operators acting on sections of a globally defined line bundle, but (\ref{Llambda2}) remains valid in
this more general context. Now that $\cD_\l$ has been defined for any $\l\in\fh^*$, the assumption of the simple
connectivity of $G$ has served its purpose and can be dropped.

\begin{thm}[Beilinson-Bernstein \cite{bb:1981}]\label{thm:Ulambda} For any $\l\in\fh^*$, the morphism {\rm(\ref{Llambda2})} is an
isomorphism of filtered algebras; thus $H^0(X,\cD_\l)=U_\l(\fg)$. Also, $H^p(X,\cD_\l)=0$ for all $p>0$ and any
$\l\in\fh^*$.
\end{thm}

This theorem is the starting point of Beilinson-Bernstein localization. We let $\Mod(U_\l(\fg))_{\text{fg}}$ denote
the category of finitely generated $U_\l(\fg)$-modules~-- or equivalently, the category of finitely generated
$U(\fg)$-modules on which $Z(\fg)$ acts via the character $\chi_\l$. Analogously $\Mod(\cD_\l)_{\text{coh}}$ shall
refer to the category of $\cD_\l$-coherent (sheaves of) $\cD_\l$-modules. Since $X$ is projective, it can be
alternatively characterized as the category of finitely generated $\cD_\l$-modules. In view of the theorem,
\begin{equation}
\label{localization1}
\begin{aligned}
\Delta \, : \, \Mod(U_\l(\fg))_{\text{fg}}\, \longrightarrow\, \Mod(\cD_\l)_{\text{coh}}\ ,\ \ \ \ \ \Delta\,M\,=\, M\otimes_{U_\l(\fg)}\cD_\l\,,
\\
\G\, : \, \Mod(\cD_\l)_{\text{coh}}\, \longrightarrow\, \Mod(U_\l(\fg))_{\text{fg}} \ ,\ \ \ \ \G\,\cM\,=\,H^0(X,\cM)\,,
\end{aligned}
\end{equation}
are well defined covariant functors. The former is called ``localization", and the latter is simply the global
section functor, of course.

Recall that $\l\in\fh_\R^*$ is dominant if $2\frac{(\l,\a)}{(\a,\a)}\geq 0$ for all $\a\in\Phi^+$. More generally
one calls $\l\in\fh^*$ dominant if its real part $\Re \l$, relative to the real structure $\fh_\R^*$, is dominant.
There is a related notion,
\begin{equation}
\label{intdom}
\l\in\fh^*\ \ \text{is integrally dominant if}\ \ \textstyle2\frac{(\l,\a)}{(\a,\a)}\notin \Z_{<0}\ \ \text{for any $\,\a\in \Phi^+$},
\end{equation}
which is less restrictive than dominance. As usual, we call $\l\in\fh^*$ regular if $(\l,\a)\neq 0$ for all
$\a\in\Phi$\,.

\begin{thm}[Beilinson-Bernstein \cite{bb:1981}]\label{thm:AB} {\rm A)}\ \ If $\l$ is regular and integrally dominant,
the stalks of any $\cM\in\Mod(\cD_\l)_{\rm{coh}}$ are generated over $\O_X$ by the global sections of
$\cM$.\newline \noindent{\rm B)}\ \ If $\l$ is integrally dominant, $H^p(X,\cM)=0$ for all
$\cM\in\Mod(\cD_\l)_{\rm{coh}}$ and any $p>0$.
\end{thm}

From this, Beilinson-Bernstein deduce, by an essentially formal argument:

\begin{cor}[Beilinson-Bernstein \cite{bb:1981}]\label{cor:AB}\ \ If $\l$ is regular and integrally dominant, the functor
$\,\Delta\,$ defines an equivalence of categories $\Mod(U_\l(\fg))_{\rm{fg}}\cong \Mod(\cD_\l)_{\rm{coh}}$, with
inverse $\Gamma$.
\end{cor}

In particular, in the case of a regular integrally dominant $\l\in\fh^*$, the irreducible objects in
$\Mod(U_\l(\fg))_{\rm{fg}}$ correspond bijectively to irreducible objects in $\Mod(\cD_\l)_{\rm{coh}}$. That much
is still true in the singular~-- i.e., non-regular~-- but still integrally dominant case, although the general
equivalence of categories fails without the regularity assumption: there may exists $\cM\in
\Mod(\cD_\l)_{\rm{coh}}$ which have no cohomology in any degree. One can then define a quotient category of
$\Mod(\cD_\l)_{\rm{coh}}$ by setting all cohomologically trivial $\cM$ equal to zero; the quotient category is
equivalent to $\Mod(U_\l(\fg))_{\rm{fg}}$. If one is interested primarily in irreducible modules~-- as we are in
the study of irreducible unitary representations of $G_\R$~-- one can deal with the singular case more simply by
realizing any irreducible $M\in\Mod(U_\l(\fg))_{\rm{fg}}$ as the space of global sections of a unique
$\cM\in\Mod(\cD_\l)_{\rm{coh}}$; as was just mentioned, one can do so provided $\l$ is at least integrally
dominant.

We shall now describe a general construction of irreducible objects in the category $\Mod(\cD_\l)_{\rm{coh}}$.
Since we are not assuming that $\l-\rho\in\L$, $\cL_{\l-\rho}$ need not exist as global $G$-equivariant line bundle
on $X$. However, it is not difficult to see that $\cL_{\l-\rho}$ does exist as an infinitesimally $\fg$-equivariant
line bundle on various (Zariski) open subsets $U\subset X$. Its sheaf of sections $\O(\cL_{\l-\rho})$ is then
irreducible as $\cD_\l$-module on $U$. We now suppose that $Q\subset X$ is a smooth subvariety, which is contained
in one of the Zariski open subsets $U$ on which $\cL_{\l-\rho}$ has meaning. The second ingredient of the
construction is ${\,\cS}\to Q$, an irreducible algebraic vector bundle on $Q$ with flat connection. Then
$\O_Q(\cS)$ is an irreducible $\cD_Q$-module, which is moreover locally free over $\O_Q$. Locally $\cD_\l$ is
isomorphic to $\cD_X$, and $\O(\cL_{\l-\rho}\!\mid_Q)$ is locally isomorphic to $\O_X$, so we can take the
$\cD_{\l}$-module direct image of $\O_Q(\cL_{\l-\rho}\!\mid_Q\otimes_\C\cS)$ under the inclusion
\begin{equation}
\label{jQX}
j\ : \ Q \ \ \hookrightarrow \ \ X\ .
\end{equation}
This requires passage to the derived category, and the operation of $\cD$-module direct image does not preserve
coherence in general. However, $\O_Q(\cL_{\l-\rho}\!\mid_Q\otimes_\C\cS)$ is locally free over $\O_Q$, hence
holonomic. The $\cD$-module direct image does preserve holonomicity, so
\begin{equation}
\label{SinX+}
j_+\O_Q(\cL_{\l-\rho}\!\mid_Q\otimes_\C\cS)\, \in \,D^b\!\left(\Mod(\cD_\l)_{\rm{hol}}\right)\,;
\end{equation}
here $D^b\left(\Mod(\cD_\l)_{\rm{hol}}\right)$ refers to the bounded derived category, consisting of complexes with
holonomic cohomology.

We let $\cH^p\!\left(j_+\O_Q(\cL_{\l-\rho}\!\mid_Q\otimes_\C\cS)\right)$ denote the cohomology sheaves of the
direct image. When restricted to $Q$, the higher direct images~-- i.e., those indexed by $p>0$~-- vanish, and the
restriction of $\cH^0\!\left(j_+\O_Q(\cL_{\l-\rho}\!\mid_Q\otimes_\C\cS)\right)$ coincides with
$\O_Q(\cL_{\l-\rho}\!\mid_Q\otimes_\C\cS)$. One calls
\begin{equation}
\label{stdmod1}
\cM(Q,\l,\cS)\ \ =_{\text{def}}\ \ \cH^0\!\left(j_+\O_Q(\cL_{\l-\rho}\!\mid_Q\otimes_\C\cS)\right)\ \in \ \Mod(\cD_\l)_{\rm{hol}}
\end{equation}
the standard sheaf corresponding to the given set of data. A proof of the following statement can be found in
\cite{milicic:notes}, for example.

\begin{prop}\label{prop:uniquesub}
The standard module $\cM(Q,\l,\cS)$ has a unique irreducible submodule $\cI(Q,\l,\cS)$. The quotient
$\cM(Q,\l,\cS)/\cI(Q,\l,\cS)$  has support on the boundary $\,\partial\,Q$; in particular, $\cM(Q,\l,\cS)$ is
irreducible and coincides with $\cI(Q,\l,\cS)$ when $Q\subset X$ is closed. Every irreducible holonomic
$\cD_\l$-module can be realized as $\cI(Q,\l,\cS)$ for some suitable choice of $Q,\,\l,\,\cS$.
\end{prop}

We should remark that the correspondence between irreducible holonomic and standard $\cD_\l$-modules is essentially
bijective: given an irreducible $\cM\in\Mod(\cD_\l)_{\rm{hol}}$\,,\,\ one can choose as $Q$ any (Zariski) open
subset of the regular set of the support of $\cM$ such that the restriction $\cM\!\mid_Q$ is locally free over
$\cD_\l$; choosing $Q$ maximally makes it unique. Combining the proposition with the equivalence of categories
\ref{cor:AB}, one finds:

\begin{cor}\label{cor:uniquesub}
If $\l\in\fh^*$ is integrally dominant, $H^0\left(X,\cI(Q,\l,\cS)\right)$ is an irreducible $U_\l(\fg)$-module or
is zero\begin{footnote}{It can vanish only when $\l$ is singular, of course.}\end{footnote}; if non-zero, it is
also the unique irreducible submodule of the $U_\l(\fg)$-module $H^0\left(X,\cM(Q,\l,\cS)\right)$.
\end{cor}

Holonomicity is not automatic for irreducible coherent $\cD_\l$-modules in general. However, as we shall discuss at
the end of this section, it does become automatic if certain additional structures are imposed~-- for example, in
the Harish Chandra setting. In settings when holonomicity does become automatic, proposition \ref{prop:uniquesub}
and its corollary almost provide a classification of the irreducible $U_\l(\fg)$-modules. What is missing is
information of when an irreducible sheaf fails to have global sections. That missing piece is provided by a result
of Beilinson-Bernstein; see \cite{HMSWII}, for example.

The usual definition of the $\cD$-module direct image, e.g. in \cite{borel:1987}, requires passage from left to
right $\cD$-modules. One can convert left to right $\cD$-modules by twisting with the canonical bundle. If one
wants to stay within the universe of left $\cD$-modules, as we have chosen to do, the definition of the direct
image requires a global twist by the ratio of the canonical bundles of $X$ and $Q$, i.e., a twist by the top
exterior power of the normal bundle $T^*_QX$. The description of the polarization in the next section depends on
an explicit description of $j_+\O_Q(\cL_{\l-\rho}\!\mid_Q\otimes_\C\cS)$, which we shall now give.

The inclusion $Q\subset X$ can be expressed as the composition of a smooth closed embedding and an open embedding:
\begin{equation}
\label{j12QX}
j\, = j_2 \circ j_1\,,\ \ \ \ j_1\,:\, Q \ \hookrightarrow \ X - \partial\,Q\,,\ \ \ \ j_2\,:\, X - \partial\,Q \ \hookrightarrow \ X\,.
\end{equation}
Thus $j_+\O_Q(\cL_{\l-\rho}\!\mid_Q\otimes_\C\cS)=
{j_2}_+\!\left(\,{j_1}_+\O_Q(\cL_{\l-\rho}\!\mid_Q\otimes_\C\cS)\right)$. Let $\,c\,$ denote the codimension of $Q$
in $X$, and
\begin{equation}
\label{DQlambda}
\cD_{Q,\l}\, =_{\text{def}}\, \O_Q(\cL_{\l-\rho}\!\mid_Q\!\otimes_\C \wedge^c\, T^*_QX)\otimes_{\O_Q}\cD_Q\otimes_{\O_Q}\O_Q(\cL_{\rho-\l}\!\mid_Q\!\otimes_\C \wedge^c\, T_QX)\ .
\end{equation}
Then $\O_Q(\cL_{\l-\rho}\!\mid_Q\otimes_\C\cS\otimes_\C \wedge^c\, T^*_QX)$ is tautologically a left
$\cD_{Q,\l}$-module. Less obviously, the sheaf restriction of $\cD_\l$ to $Q$ is a right $\cD_{Q,\l}$-module, as
can be seen by converting from right to left $\cD$-modules and back again, or equivalently by twisting with the
canonical bundle of $X$ and the inverse of the canonical bundle of $Q$. The $\cD$-module direct image under the
closed smooth embedding $j_1$ is given by
\begin{equation}
\label{j1+QX}
{j_1}_+\O_Q(\cL_{\l-\rho}\!\mid_Q\otimes_\C\cS)\ = \ \cD_\l\otimes_{\cD_{Q,\l}}\O_Q(\cL_{\l-\rho}\!\mid_Q\otimes_\C\cS\otimes_\C \wedge^c\, T^*_QX)\,.
\end{equation}
This step does not require going to the derived category, i.e., ${j_1}_+\O_Q(\cL_{\l-\rho}\!\mid_Q\otimes_\C\cS)$
exists as a coherent $\cD_\l$-module. At the next step, the $\cD$-module direct image ${j_2}_+$ under the open
embedding $j_2$ coincides with the derived sheaf direct image $R^{\,\raisebox{.4ex}{\bf .}}{j_2}_*\,$,
\begin{equation}
\label{j2+QX}
\begin{aligned}
{j}_+\O_Q(\cL_{\l-\rho}\!\mid_Q\otimes_\C\cS)\ &= \ {j_2}_+\!\left({j_1}_+\O_Q(\cL_{\l-\rho}\!\mid_Q\otimes_\C\cS)\right)
\\
&= \ R^{\,\raisebox{.4ex}{\bf .}}{j_2}_*\!\left({j_1}_+\O_Q(\cL_{\l-\rho}\!\mid_Q\otimes_\C\cS)\right)\,.
\end{aligned}
\end{equation}
It follows that the standard $\cD_\l$-module $\cM(Q,\l,\cS)$\,,
\begin{equation}
\label{stdmod2}
\cM(Q,\l,\cS) = \cH^0\!\left(j_+\O_Q(\cL_{\l-\rho}\!\mid_Q\otimes_\C\cS)\right) =  R^{\,0}{j_2}_*\!\left({j_1}_+\O_Q(\cL_{\l-\rho}\!\mid_Q\otimes_\C\cS)\right),
\end{equation}
is the ordinary, underived sheaf direct image of the sheaf (\ref{j1+QX}) under the open embedding $j_2$\,. This is
the description we shall need.

As was mentioned in section \ref{sec:hermitian forms}, a Harish Chandra module with infinitesimal character
$\chi_\l$ is a finitely generated $U_\l(\fg)$-module, with a compatible structure of algebraic $K$-module. The
analogous notion on the level of $\cD$-modules is that of a Harish Chandra sheaf: a finitely generated
$\cD_\l$-module on $X$, equipped with the structure of $K$-equivariant quasi-projective $\O_X$-module, such that
the two structures are compatible. In this context, compatibility means that the infinitesimal action of $\fk$ on
the sheaf, by differentiation of the $K$-action, coincides with the action of $\fk$ via its inclusion in $\fg$ and
the homomorphism (\ref{Llambda2}). We let $\HC(\cD_\lambda,K)$ denote the category of Harish Chandra sheaves and
$(\cD_\lambda,K)$-homomorphisms between them\begin{footnote}{If $G_\R$ and hence also $K$ are connected, the
$K$-action on a Harish Chandra sheaf is completely determined by the $\fk$-action, and hence by the $\cD_\l$-module
structure. Any $\cD_\l$-morphism between Harish Chandra sheaves is then a morphism in the category
$\HC(\cD_\lambda,K)$.}\end{footnote}. Feeding the $K$-action into the equivalence of categories \ref{cor:AB}, one
obtains:

\begin{cor}\label{cor:ABHC} \ \ If $\l$ is regular and integrally dominant, the functor
$\,\Delta\,$ defines an equivalence of categories $\HC(\fg,K)_\l\cong \HC(\cD_\l,K)$, with inverse functor
$\Gamma$. If $\l$ is integrally dominant, every irreducible Harish Chandra module can be realized as the space of
global sections $H^0(X,\cM)$ of a unique irreducible Harish Chandra sheaf $\cM$.
\end{cor}

By definition, a standard Harish Chandra sheaf is one constructed as in the discussion leading up to (\ref{SinX+}),
but with ingredients $Q$, $\cS$ that make the resulting sheaf $K$-equivariant: we require $Q\subset X$ to be one of
the~-- finitely many~-- $K$-orbits in $X$, and $\cS$ an irreducible $K$-equivariant flat vector bundle over $Q$. In
other words, $\cS\to Q$ is a $K$-homogeneous vector bundle associated to an irreducible representation of the
component group of the isotropy subgroup of $K$ at any particular point of $Q$; the component group is finite
abelian, so $\cS$ necessarily has rank one. Implicit in the construction is the assumption that $\cL_{\l-\rho}$
exists as infinitesimally $\fg$-equivariant line bundle on some neighborhood of $Q$, which turns out to be
equivalent to a partial integrality condition on $\l-\rho$; for details see \cite{HMSWI}. The $K$-equivariant
ingredients $Q$ and $\cS$ make the sheaf $\O_Q(\cL_{\l-\rho}\!\mid_Q\otimes_\C\cS)$ $K$-equivariant, compatibly
with its structure as module over the appropriate $K$-equivariant twisted sheaf of differential operators on $Q$.
The $\cD_\l$-direct image inherits a compatible $K$-equivariant structure by functorality.

Standard Harish Chandra sheaves are all those which arise in this manner, from a $K$-orbit $Q\subset X$ and an
irreducible $K$-equivariant flat vector bundle $\cS$. Coherent sheaves of $\cD_\l$-modules, equivariant with
respect to a subgroup of $G$ which acts on $X$ with finitely many orbits, are automatically regular holonomic. In
particular this applies to Harish Chandra sheaves. Thus, as a special case of proposition \ref{prop:uniquesub}, one
obtains:

\begin{prop}\label{prop:HCclassification} \ \ Irreducible Harish Chandra sheaves correspond bijectively to standard
Harish Chandra sheaves, via inclusion of the former in the latter, as the unique irreducible subsheaf. Standard
sheaves associated to closed $K$-orbits are irreducible.
\end{prop}

We mentioned earlier that a criterion of Beilinson-Bernstein characterizes those irreducible sheaves which have no
cohomology. In combination with this criterion, corollary \ref{cor:ABHC} and proposition
\ref{prop:HCclassification} provide a classification of the irreducible Harish Chandra modules. For details see
\cite{HMSWII}, which also connections the various different classification schemes and relates algebraic properties
of Harish Chandra sheaves to analytic properties~-- such as square integrability and temperedness~-- of the
corresponding $G_\R$-representations.

The flag variety $X$ is the universal complex projective variety with a transitive algebraic $G$-action~--
universal in the sense that any other such variety arise as a $G$-equivariant image of $G$. These $G$-equivariant
images are called generalized flag varieties. Everything that has been said so far can be adapted to the setting of
generalized flag varieties. However, for a generalized flag variety, the sheaves of twisted differential operators
are defined not for all $\l\in\fh^*$, but only for $\l$ in a suitable subspace of $\fh^*$. For that reason, only
certain Harish Chandra modules can be constructed from Harish Chandra sheaves on any particular generalized flag
variety. Also, $K$-orbits in $X$ affinely embedded, as was observed by Beilinson-Bernstein~-- see \cite{HMSWI}~--
but this is not true in the context of generalized flag varieties. In the general construction of standard modules
(\ref{stdmod1}), the higher direct images $\cH^p\!\left(j_+\O_Q(\cL_{\l-\rho}\!\mid_Q\otimes_\C\cS)\right)$,
$\,p>0$\,, vanish when $Q\subset X$ is affinely embedded. In that case, then,
\begin{equation}
\label{stdmod3}
\cM(Q,\l,\cS)\ = \ j_+\O_Q(\cL_{\l-\rho}\!\mid_Q\otimes_\C\cS)\ \in \ \Mod(\cD_\l)_{\rm{hol}}\ .
\end{equation}
This is a significant property of standard Harish Chandra sheaves on the flag variety, which cannot be extended to
generalized flag varieties.\vspace{2pt}

\section{Mixed Hodge modules on the flag variety}\label{sec:flag variety}

The discussion of the previous section applies to any $\l\in\fh^*$. When we put the structure of complex mixed
Hodge module on standard Harish Chandra sheaves, it becomes necessary to assume $\l\in\fh^*_\R$. The passage from
the Hodge filtration on a standard Harish Chandra sheaf to the Hodge filtration on the corresponding Harish Chandra
module involves a vanishing theorem, which depends on the dominance of $\l$~-- integral dominance is not enough.
For that reason we reimpose the condition (\ref{realinfchar2}), i.e., $\l$ is required both to lie in $\fh_\R^*$
and to be dominant.

Our application of Hodge modules involves mixed Hodge modules built from filtered $\cD_\l$-modules, rather than
from filtered $\cD$-modules as in section \ref{sec:MHM_summary}. Twisting the sheaf $\cD$ by a nowhere vanishing
function does not affect the degree filtration, and locally $\cD_\l \cong \cD$ by means of such a twist. That makes
the passage from $\cD$ to $\cD_\l$ harmless as far as the filtrations are concerned.

The other crucial ingredient of the construction of complex Hodge modules is the polarization. We shall first
describe the polarization in the untwisted case, and then discuss the necessary modification for the twisted case.
We let $\overline X$ denote $X$ equipped with the conjugate algebraic structure. Thus $f \mapsto \bar f$ defines a
conjugate linear isomorphism $\O_X \cong \O_{\overline X}$. Similarly $\cM \ \mapsto \overline{\cM}$ describes a
bijection between regular holonomic $\cD_X$-modules and regular holonomic $\cD_{\overline X}$-modules. Saito
defines the polarization on an irreducible holonomic $\cD$-module $\cM\,$,\,\ when it exists, as an isomorphism
between $\cM$ and its conjugate dual. We shall use a more concrete definition, due to Sabbah \cite{sabbah:2005},
which is based on ideas of Kashiwara \cite{kashiwara:1987}: a nondegenerate hermitian pairing
\begin{equation}
\label{spolarization1}
P\,:\,\cM \times \overline{\cM} \ \longrightarrow \ \cC^{-\infty}(X_\R)\ , \ \ \ \text{bilinear over $\cD_X\times \cD_{\overline X}$}\ ;
\end{equation}
here $X_\R$ denotes $X$ considered as $C^\infty$ manifold, and $\cC^{-\infty}(X_\R)$ the sheaf of distributions on
$X_\R$. The important point is that this notion behaves well with respect to the standard operations on regular
holonomic $\cD$-modules.

In the setting of $\cD_\l$-modules, $\cM \mapsto \overline{\cM}$ defines a bijection
$\Mod(\cD_{X,\l})_{\text{rh}}\cong \Mod(\cD_{\overline X,-\l})_{\text{rh}}$\,;\,\ the switch from $\l$ to $-\l$ is
explained by the fact that the character $e^{-\l}$ is the complex conjugate of $e^\l$. The subscripts $X$ for $\cD_{X,\l}$ and
$\overline X$ for $\cD_{\overline X,-\l}$ were necessary to make the meaning unambiguous. For simplicity, from now
on, we shall return to the notation $\cD_{ \l}$ for $\cD_{X,\l}$ and write $\overline{\cD}_{-\l}$ instead of
$\cD_{\overline X,-\l}$\,. The analogue of (\ref{spolarization1}) in the current situation is a hermitian pairing
\begin{equation}
\label{spolarization2}
P\,:\,\cM \times \overline{\cM} \ \longrightarrow \ \cC^{-\infty}(X_\R)\ , \ \ \ \text{bilinear over $\cD_\l\times \overline{\cD}_{-\l}$}\ ,
\end{equation}
for $\cM\in \Mod(\cD_{\l})_{\text{rh}}$\,.

The definition of the categories $\CHM(Z)$, $\CMHM(Z)$, $D^b(\CMHM(Z))$ in section \ref{sec:MHM_summary} carries
over to the twisted setting. If $Q\subset X$ is a smooth subvariety such that $\cL_{\l-\rho}$ exists as an
infinitesimally $\fg$-equivariant line bundle on some neighborhood of $Q$, it makes sense to consider the analogous
categories $\CHM(Q)_\l$, $\CMHM(Q)_\l$, $D^b(\CMHM(Q)_\l)$ whose objects have underlying filtered regular holonomic
$\cD_{Q,\l}$-modules; cf. (\ref{DQlambda}). If one trivializes $\cL_{\l-\rho}\!\mid_Q$ locally on $Q$, these three
categories become locally isomorphic to their untwisted counterparts.

Going back to the definition (\ref{stdmod1}) of the standard $\cD_\l$-module $\cM(Q,\l,\cS)$, we now require the
irreducible flat vector bundle $\,\cS \to Q$ to come equipped with a flat hermitian metric. In that case,
\begin{equation}
\label{spolarization3}
\begin{aligned}
&P\,:\,\O_Q(\cL_{\l-\rho}\!\mid_Q\otimes_\C\cS) \times \overline{\O_Q(\cL_{\l-\rho}\!\mid_Q\otimes_\C\cS)} \ \longrightarrow \ \cC^{-\infty}(Q_\R)\ ,
\\
&\qquad\qquad P(\sigma,\overline\tau)\ = \ \langle\,\sigma\,,\, \overline\tau\,\rangle\, \in\, \cC^{\infty}(Q_\R)\subset \cC^{-\infty}(Q_\R)\ ,
\end{aligned}
\end{equation}
defines a polarization for the irreducible $\cD_{Q,\l}$-module $\O_Q(\cL_{\l-\rho}\!\mid_Q\otimes_\C\cS)$; here
$\langle\,\cdot\ ,\,\cdot\, \rangle$ denotes the pointwise hermitian form determined by the flat hermitian metric
on $\cS$ and the intrinsic metric\begin{footnote}{The hypothesis $\l\in\fh^*_\R$ implies that
$\overline{e^{\l-\rho}}=e^{\rho-\l}$, so the fiber $\,\overline{\cL_{\l-\rho}}\!\mid_x$\,\ at any $x\in Q$ is
canonically dual to $\,\cL_{\l-\rho}\!\mid_x\,$.}\end{footnote} on $\cL_{\l-\rho}\!\mid_Q$\,. We can then turn
$\O_Q(\cL_{\l-\rho}\!\mid_Q\otimes_\C\cS)$ into an object in $\CHM(Q)_\l$ by imposing the ``trivial" Hodge and
weight filtrations
\begin{equation}
\label{stdchm1}
\begin{aligned}
F_p\,\O_Q(\cL_{\l-\rho}\!\mid_Q\otimes_\C\cS)\ &= \ \begin{cases}\ 0\ ,\ \ &p<0\ ,
\\
\ \O_Q(\cL_{\l-\rho}\!\mid_Q\otimes_\C\cS)\ ,\ \ \ &p\geq 0\ ,\end{cases}
\\
W_k\, \O_Q(\cL_{\l-\rho}\!\mid_Q\otimes_\C\cS) \ &= \ \begin{cases}\ 0\ ,\ \ &k<\dim Q\ ,
\\
\ \O_Q(\cL_{\l-\rho}\!\mid_Q\otimes_\C\cS)\ ,\ \ &k\geq \dim Q\ .\end{cases}
\end{aligned}
\end{equation}
With a slight abuse of notation, we shall write $\O_Q(\cL_{\l-\rho}\!\mid_Q\otimes_\C\cS)\in\CHM(Q)_\l$\,.\,\ Here,
as in the following, it shall be understood that the specified Hodge and weight filtration constitute part of the
structure.

Recall the definition of the morphisms $j_1$ and $j_2$ in (\ref{j12QX}). The direct image of
$\O_Q(\cL_{\l-\rho}\!\mid_Q\otimes_\C\cS)$ under the smooth closed embedding $j_1$ remains irreducible, hence pure,
\begin{equation}
\label{stdchm2}
{j_1}_+\O_Q(\cL_{\l-\rho}\!\mid_Q\otimes_\C\cS)\, \in \, \CHM(X-\partial\,Q)_\l\ ,
\end{equation}
but the open embedding $j_2$ may not preserve purity and requires passage to the derived category, so
\begin{equation}
\label{stdchm3}
{j_2}_+\left({j_1}_+\O_Q(\cL_{\l-\rho}\!\mid_Q\otimes_\C\cS)\right)\, \in \,  D^b(\CMHM(X)_\l)\ .
\end{equation}
The cohomology sheaf in degree zero inherits this structure:
\begin{equation}
\label{stdchm4}
\cM(Q,\l,\cS)\ = \ \cH^{\,0}\!\left({j_2}_+\circ {j_1}_+\O_Q(\cL_{\l-\rho}\!\mid_Q\otimes_\C\cS)\right)\, \in \,\CMHM(X)_\l\ .
\end{equation}
The unique irreducible subsheaf $\cI(Q,\l,\cS)$ of $cM(Q,\l,\cS)$ must coincide with the lowest graded piece of the
weigh filtration, which implies
\begin{equation}
\label{stdchm5}
\cI(Q,\l,\cS)\, \in \,\CHM(X)_\l\ .
\end{equation}
The same functorial arguments apply to Beilinson-Bernstein's maximal extension of
$\O_Q(\cL_{\l-\rho}\!\mid_Q\otimes_\C\cS)$ \cite{bb:1993}. At this point we have put the structure of,
respectively, Hodge module and mixed Hodge module on the irreducible $\cD_\l$-module $\cI(Q,\l,\cS)$, the standard
$\cD_\l$-module $\cM(Q,\l,\cS)$, its dual, and the maximal extension. The natural inclusion among these are
morphism of mixed Hodge modules by construction.

It is entirely possible that the flat vector bundle $\cS$ admits a non-trivial polarized variation of Hodge
structure. If so, we could have carried out the constructions above starting with the resulting Hodge filtration
instead of (\ref{stdchm1}). However, only the trivial structure exists universally and is canonical. As was pointed
out before, $\cS$ has rank one in the $K$-equivariant Harish Chandra setting. In that case, the structure
(\ref{stdchm1}) is truly unique. Recall that the hypothesis (\ref{realinfchar2}) remains in force.

\begin{thm}\label{thm:vanishing}
The global section functor $\,\Gamma : \CMHM(X)_\l \,\to\,U_\l(\fg)\,$ is exact.
\end{thm}

We shall prove this result in a future paper. As an immediate consequence, the spaces of global sections of
$\cI(Q,\l,\cS)$, of $\cM(Q,\l,\cS)$, of the dual of the latter, and of the maximal extension carry canonical,
functorial weight and Hodge filtrations.

By functorality, the irreducible $\cD_\l$-module $\cI(Q,\l,\cS)$ has a polarization, with a distinguished choice of
sign fixed by the polarization (\ref{spolarization3}) for $\O_Q(\cL_{\l-\rho}\!\mid_Q\otimes\cS)$. Since $U_\R$ is
compact, there exists an essentially unique positive $U_\R$-invariant measure $dm$ on $X$. If $\sigma,\,\tau$ are
global sections of $\cI(Q,\l,\cS)$, we can integrate the distribution $P(\sigma,\tau)$ against the smooth measure
$dm$, which results in a hermitian form on $\,\G\,\cI(Q,\l,\cS)$.

\begin{prop}\label{prop:invariance}
The hermitian form $\,(\sigma,\tau)_{\fu_\R}=\! \displaystyle\int_X \!P(\sigma,\overline\tau)\,dm\,$ on the space\vspace{-4pt} of sections $\,\G\,\cI(Q,\l,\cS)$ is
$\,\fu_R$-invariant.\vspace{-2pt}
\end{prop}

\begin{proof}
We identify $\fu_\R \subset \ \fu_\R \otimes \C = \fg \subset \fg \oplus \overline\fg$ with a Lie algebra of real
vector fields on $X$ by infinitesimal translation, and correspondingly $\fg$ and $\overline\fg$ on the right of
this inclusion\begin{footnote}{Via $\fg \ni \zeta \mapsto (\zeta,\overline\zeta)\in
\fg\oplus\overline\fg$\,.}\end{footnote} with Lie algebras of, respectively, holomorphic and antiholomorphic vector
fields. We can then express any $\zeta\in\fu_\R$ as a sum $\zeta=\xi + \overline{\xi}$, with $\xi\in\fg$. Since
$P(\ ,\ )$ is $\cD_\l\times \overline{\cD}_{-\l}$-bilinear, and since $\overline\fg$ and $\fg$ annihilate,
respectively, $\cI(Q,\l,\cS)$ and $\overline{\cI(Q,\l,\cS)}$,
\begin{equation}
\begin{aligned}
\label{spolarization4}
(\zeta\sigma,\tau)_{\fu_\R} + (\sigma,\zeta\tau)_{\fu_\R} \ &= \ \int_X P(\xi\sigma,\overline\tau)\,dm\ + \ \int_X P(\sigma,\overline\xi\overline\tau)\,dm\,
\\
&= \ \int_X (\xi+\overline\xi)\left(P(\sigma,\overline\tau)\right)\,dm\ = \ \int_X \zeta\!\left(P(\sigma,\overline\tau)\right)\,dm
\\
&= \ \int_X \!\zeta\left(P(\sigma,\overline\tau)\,dm\right)\ \ =\ \ 0\ .
\end{aligned}
\end{equation}
At the final step we have used the $\fu_\R$-invariance of $dm$ and the fact that an infinitesimal translate of a
differential form of top degree is exact.
\end{proof}

In the Harish Chandra setting, the Cartan involution $\theta$ acts on the flag variety. The irreducible Harish
Chandra module $M = \Gamma\,\cI(Q,\l,\cS)$ carries a nontrivial $\fg_\R$-invariant hermition form only when
$\theta$ fixes the geometric data $(Q,\l,\cS)$. In that case $\theta$ acts on the sheaf $\,\cI(Q,\l,\cS)$ and its
space of global sections $M$. It is not difficult to show that action of $\theta$ coincides with the operator $T$
of proposition \ref{prop:urgr}, which relates the $\fg_\R$-invariant hermitian form to the $\fu_\R$-invariant one.

We can now state our conjecture more generally, not only for Harish Chandra modules as in (\ref{conj1}):

\begin{conj}\label{conj2}
Suppose $M = \Gamma\,\cI(Q,\l,\cS)$ is an irreducible $U_\l(\fg)$-module cor\-res\-ponding to the geometric data
$Q,\,\l,\,\cS$, as described earlier in this section. Then $v \in F_p M\cap (F_{p-1} M)^\perp,$ $\,v\neq 0\,$,
implies $\,(-1)^{p-c}(v,v)_{\fu_R}
> 0\,$ for all $p$\,.
\end{conj}

The integer $c=\codim_XQ$ turns out to be the lowest index of the Hodge filtration, which was denoted by $\,a\,$ in
conjecture \ref{conj1}.

We have verified the conjecture for irreducible Harish Chandra modules in quite a few cases. The simplest example
of an irreducible $\cD_\lambda$-module not in Harish Chandra's category is
\begin{equation}
\label{verma}
V_\l\ = \ \cI(\{\text{pt}\},\l,\C)\ = \ \cM(\{\text{pt}\},\l,\C)\, ,
\end{equation}
the $\cD_\l$-module direct image of the ``flat vector bundle" $\C$ over a point -- in other words, the irreducible
Verma module of lowest weight $\l+\rho$. In that case, the hermitian form $(\ ,\ )_{\fu_\R}$ coincides with the
Shapolvalov form. Its signature character was computed by Wallach \cite{wallach:1984}, whose formula is compatible
with conjecture \ref{conj2} for $V_\l$ but does not imply it. We did verify the conjecture in this particular case.
It seems likely that this fact will be relevant for the proof of the general conjecture.

We should mention that the extension of the polarization (\ref{spolarization3}), via ${j_1}_+$\,,\,\ from
$\O_Q(\cL_{\l-\rho}\!\mid_Q\otimes_\C\cS)$ to ${j_1}_+\left(\O_Q(\cL_{\l-\rho}\!\mid_Q\otimes_\C\cS)\right)$ is
easy to describe. Recall the identity (\ref{j1+QX}). If $\sigma$, $\tau$ are local sections of
$\O_Q(\cL_{\l-\rho}\!\mid_Q\otimes_\C\cS\otimes \wedge^cT^*_QX)$ and $D_1$, $D_2$ local sections of $\cD_\l$, all
with overlapping domains, then $D_1\sigma$, $D_2\tau$ can be regarded as as local sections of
$\cD_\l\otimes_{\cD_{Q,\l}}\O_Q(\cL_{\l-\rho}\!\mid_Q\otimes_\C\cS\otimes
\wedge^cT^*_QX)={j_1}_+\!\left(\O_Q(\cL_{\l-\rho}\!\mid_Q\!\otimes_\C\cS)\right)$, and
\begin{equation}
\label{spolarization5}
\begin{aligned}
P\,:\,{j_1}_+\!\left(\O_Q(\cL_{\l-\rho}\!\mid_Q\!\otimes_\C\cS)\right)\! \times\! \overline{{j_1}_+\!\left(\O_Q(\cL_{\l-\rho}\!\mid_Q\!\otimes_\C\cS)\right)}  \to  \cC^{-\infty}(X_\R- \partial Q_\R)\,,
\\
P(D_1\sigma,\overline{D_2\tau}) = D_1\,\overline{D_2}\,\langle\,\sigma\,,\, \overline{\tau}\,\rangle\, \in\, \cC^{\infty}(Q_\R,\wedge^{2c}T^*_{Q_\R}X_\R)\subset \cC^{-\infty}(X_\R-\partial Q_\R)\,,
\end{aligned}
\end{equation}
describes the polarization for ${j_1}_+\left(\O_Q(\cL_{\l-\rho}\!\mid_Q\otimes_\C\cS)\right)$. Note that in this
identity, $\sigma$ and $\overline\tau$ contract to a section of the top exterior power of the conormal bundle
$T^*_{Q_\R}X_\R$ of the real submanifold $Q_\R \subset X_\R$. Any such section represents a distribution on
$X_\R-\partial Q_\R$ with support in $Q_\R$\,,\,\ since it can be paired against any smooth measure on
$X_\R-\partial Q_\R$, resulting in a form of top degree on $Q_\R$ with distribution coefficients~-- a smooth form
of top degree, in fact~-- which could be integrated over $Q_\R$ if it were compactly supported. As a distribution
on $X_\R-\partial Q_\R$\,,\,\ $\langle\,\sigma\,,\, \overline{\tau}\,\rangle$ can be acted upon by the differential
operators $D_1$, $\overline D_2$.

The direct image of the polarizarion (\ref{spolarization5}) under the open embedding $j_2$ becomes the integrand in
the definition (\ref{prop:invariance}) of the hermitian form $(\ ,\ )_{\fu_\R}$. In a future paper we shall give an
explicit description of the $\cD$-module direct image under an open embedding. That description provides a handle
on the hermitian form $(\ ,\ )_{\fu_\R}$. \vspace{2pt}

\section*{A.\,\ Appendix}\label{sec:appendix}
\renewcommand{\thesection}{\Alph{section}}\setcounter{section}{1}\setcounter{equation}{0}

We shall briefly comment on the extension of Saito's theory we need for our applications. In one way, the
development of the theory of mixed Hodge modules is formally analogous to that of the theory of regular holonomic
$\cD$-modules, though the former is far more intricate, of course: in both cases the existence of a category
satisfying certain formal properties is reduced to the case of a base space of dimension one. For the one
dimensional case, Saito relies on the nilpotent orbit theorem and $SL_2$-orbit theorem of \cite{schmid:1973},
though in later pages he refers only to the multi-variable generalization \cite{CKS:1986}. The extension of Saito's
theory depends on a relaxation of the hypotheses of \cites{CKS:1986,schmid:1973}.

The two papers describe, in very precise terms, the degeration of a polarized variation of Hodge structure~-- over
a one dinensional base in \cite{schmid:1973}, and over the complement of a divisor with normal crossings, near a
point of the divisor, in \cite{CKS:1986}. In both cases the existence of an underlying rational structure implies
that the monodromy eigenvalues are roots of unity. After \cite{schmid:1973} was written, Deligne observed that the
arguments go through as long as the eigenvalues have absolute value one. The same is true for the results of
\cite{CKS:1986}, which are based on those of \cite{schmid:1973}. In the terminology of section
\ref{sec:MHM_summary}, this means that \cites{CKS:1986,schmid:1973}, with minor modifications, describe also the
degeneration of complex polarized variations of Hodge structure; the polarization forces the eigenvalue one
property.

Taking the complex conjugate of a complex mixed Hodge structure with its complex conjugate results in a mixed Hodge
structure with underlying real structure~-- rather than an underlying rational structure, as is commonly assumed.
This resolves the seeming contradiction commented upon in section \ref{sec:MHM_summary}, following
\ref{singleton*}. The only critical use of the underlying rational structure in Saito's theory involves the orbit
theorems; for other purposes an underlying real structure suffices. In this way one can avoid relying on an
underlying rational structure, and via the trick of taking the direct sum with the complex conjugate, even on a
real structure.

\end{document}